\documentclass[a4paper,11pt]{article}
\usepackage[all]{xy}
\usepackage[T1]{fontenc}
\usepackage[utf8]{inputenc}
\usepackage{lmodern}
\usepackage{microtype}
\usepackage{mathtools}
\usepackage{amsfonts}
\usepackage{amsmath}
\usepackage{amssymb}
\usepackage{pictexwd,dcpic}
\usepackage{thmtools}
\newtheorem{theorem}{Theorem}[section]
\newtheorem{lemma}[theorem]{Lemma}
\newtheorem{proposition}[theorem]{Proposition}
\newtheorem{corollary}[theorem]{Corollary}

\newenvironment{proof}[1][Proof]{\begin{trivlist}
\item[\hskip \labelsep {\bfseries #1}]}{\end{trivlist}}
\newenvironment{definition}[1][Definition]{\begin{trivlist}
\item[\hskip \labelsep {\bfseries #1}]}{\end{trivlist}}

\newenvironment{remarks}[1][Remarks]{\begin{trivlist}
\item[\hskip \labelsep {\bfseries #1}]}{\end{trivlist}}
\newenvironment{remark}[1][Remark]{\begin{trivlist}
\item[\hskip \labelsep {\bfseries #1}]}{\end{trivlist}}

\newenvironment{conjecture}[1][Conjecture]{\begin{trivlist}
\item[\hskip \labelsep {\bfseries #1}]}{\end{trivlist}}

\newcommand{\qed}{\nobreak \ifvmode \relax \else
      \ifdim\lastskip<1.5em \hskip-\lastskip
      \hskip1.5em plus0em minus0.5em \fi \nobreak
      \vrule height0.75em width0.5em depth0.25em\fi}
\usepackage[top=3cm,bottom=3cm,left=2.5cm,right=3cm]{geometry}

 \DeclareFontFamily{U}{wncy}{}
    \DeclareFontShape{U}{wncy}{m}{n}{<->wncyr10}{}
    \DeclareSymbolFont{mcy}{U}{wncy}{m}{n}
    \DeclareMathSymbol{\Sh}{\mathord}{mcy}{"58}
\usepackage{hyperref}
\usepackage{fancyhdr}
\date{}
  \title{On the existence of special elements in odd $K$-theory groups in finite abelian extensions of imaginary quadratic fields}
  \author{Jilali Assim$^1$ and Saad El Boukhari$^2$}

\begin{document}
\maketitle
\date{}
\begin{abstract}
Let $k$ be an imaginary quadratic number field, and $F/k$  a finite abelian extension of  Galois group $G$.
We investigate the relationship between the conjectural special elements introduced in \cite{Burns-DeJeu-Gangl} and ETNC in the semi-simple case. This provides a partial proof of the conjecture for $F/k$ under certain conditions.
\end{abstract}


\section{Introduction}
Let $K$ be a number field and $X$  a smooth projective variety over $K$. Let $i$ and $j$ denote integers in $\mathbb{Z}$. We consider the motive $M=h^{i}(X)(j)$. For practical purposes, one needs only to understand the realisations of $M$:
\begin{itemize}
\item For any rational prime $p$, the étale realisation $M_p:=H^{i}_{ét}(X\times_{K}\overline{K}, \mathbb{Q}_p)(j)$ which is a continuous representation of the Galois group $\mathrm{Gal}(\overline{K}/K)$.
\item The Betti realisation $M_B:=H^{i}(X(\mathbb{C}),\mathbb{Q})(j)$ which is a finite dimentional space and carries an action of complex conjugation.
\item The deRham realisation $M_{dR}:=H^{i}_{dR}(X/\mathbb{Q})(j)$ which is a finite dimensional filtered space.
\item We also attach to $M$ the motivic cohomology spaces $H^{0}_{\mathcal{M}}(M)$ and $H^{1}_{\mathcal{M}}(M)$ which are (conjecturally) finite dimentional.
\end{itemize}
For each such motive $M$, there is a dual motive $M^{\vee}$ with dual realisations. The period isomorphism $M_{B}\otimes\mathbb{C}\cong M_{dR}\otimes\mathbb{C}$ induces a map
$$\alpha_{M}: M_{B}^{+}\otimes\mathbb{R}\rightarrow M_{dR}/\mathrm{Fil}^{0}M_{dR}\otimes \mathbb{R},$$
and it is conjecured that we have the following
\begin{conjecture}{(c.f. \cite{Fontaine})}
There is an exact sequence
\begin{align*}
0\rightarrow H^{0}_{\mathcal{M}}(M)\otimes\mathbb{R}\rightarrow \mathrm{ker}(\alpha_{M})\stackrel{\rho_b^{\vee}}{\rightarrow} H^{1}_{\mathcal{M}}(M^{\vee}(1))^{\vee}\otimes \mathbb{R}&\rightarrow \\ H^{1}_{\mathcal{M}}(M)\otimes\mathbb{R}&\stackrel{\rho_b}{\rightarrow}\mathrm{coker}(\alpha_{M})\rightarrow H^{0}_{\mathcal{M}}(M^{\vee}(1))^{\vee}\otimes \mathbb{R}\rightarrow 0
\end{align*}
where $\rho_b$ is the Beilinson regulator. 
\end{conjecture}
Motivic cohomology is conjecturally closely related to the special values of the $L$-function $L(M, s)$ associated with the motive $M$, and defined for $\mathrm{Re}(s)$ large enough, by the following product over rational primes $\ell$
$$L(M, s):=\prod_{\ell}\mathrm{det}(1-\mathrm{Fr}_\ell^{-1}T\mid M_p^{I_\ell})$$
where $Fr_\ell$ is the Frobenius element. In the most basic example, when $M= h^{0}(\mathrm{Spec}(K))(0)$, $H^{1}_{f}(M^{\vee}(1))\simeq O_{K}^{\times}\otimes \mathbb{Q}$ and $M_{B}\otimes \mathbb{R}\simeq \bigoplus_{\sigma: K\rightarrow \mathbb{C}} \mathbb{R}$.  The Beilinson regulator is in this case the Dirichlet regulator $R^{D}_F$ defined as follows
\begin{align*}
R^{D}_{K}: O_{K}^{\times}\otimes \mathbb{R} &\rightarrow 
\bigoplus_{\sigma: F\rightarrow \mathbb{C}}\mathbb{R}\\
u &\mapsto \Sigma_{\sigma: F\rightarrow \mathbb{C}}
(\mathrm{log} \mid\sigma(u)\mid )\sigma
\end{align*}
where $\mid\sigma(u)\mid:=(\sigma(u)\overline{\sigma(u)})^{1/2}$. Suppose now that $K$ is abelian and let $G:=\mathrm{Gal}(K/\mathbb{Q})$. The $L$-function $L(M,s)$ is a $\mathbb{C}[G]$-valued function, and can be identified with
$$L(M, s)=(L(\chi, s))_{\chi\in \mathrm{Hom}(G, \mathbb{C})},$$
where, $L(\chi, s)$ is the usual Dirichlet $L$-function associated with the character $\chi$. If we denote by $L^{*}(\chi, 0)$ the special value of the $L$-function at $s=0$, then
$$L^{*}(\chi, 0)\in \mathbb{R}^{\times}$$
We can make the regulator $R^{D}_{K}$ map into $\mathbb{R}$ by slightly changing its definition: We can redefine $R^{D}_{K}$ over $(O_{K}^{\times})^{r_1+r_2}$, where $r_1$ (resp. $r_2$) is the number of real embeddings (resp. pairs of complex conjugate embeddings) of $K$, as the determinant in $\mathbb{R}$ of the real numbers $(\mathrm{log}(\mid \sigma_{i}(a_j)\mid))_{1\leq i,j\leq r_1+r_2}$, where $(a_j)_{1\leq j\leq r_1+r_2}\in (O_{K}^{\times})^{r_1+r_2}$.\\
The question is then whether or not the special value $L^{*}(\chi, 0)$ lies in the image of $(O_{K}^{\times})^{r_1+r_2}$ by $R^{D}_{K}$. The answer comes obviously, in this case, with the definition of cyclotomic units (e.g. \cite{washington}, Chap. 8).\\
Suppose now that $F/k$ is a finite abelian extension of number fields with Galois group $G$, and  $M=h^{0}(\mathrm{Spec}(F))(j)$ (The motive is over $k$), with $j=r<0$. The map $\rho_b$ is in this case the usual $K$-theory Beilinson regulator
$$\rho_b=\rho^{r}_{F}: K_{1-2r}(F)\otimes\mathbb{R}\rightarrow (\prod_{\sigma: F\rightarrow \mathbb{C}}(2\pi i)^{-r}\mathbb{R})^{+}.$$
and the motivic $L$-function is identified with
$$L(M, s)=(L(\chi, s+r))_{\chi\in\mathrm{Hom}(G,\mathbb{C})},$$
where $L(\chi, s)$ is the Artin $L$-function.
Since  $L^{*}(\chi, r)\in\mathbb{R}^{\times}$, we can ask again if one can map a certain element in $K_{1-2r}(F)$ to a certain image involving $L^{*}(\chi, r)$ by the regulator $\rho^{r}_{F}$.\\
In the case where $k=\mathbb{Q}$, and $F=\mathbb{Q}(\zeta_{N})$ ($N>0$) and $\zeta_{N}=e^{2\pi i/N}$, construction of  such element $\ell(w)$ in $K_{1-2r}(F)\otimes \mathbb{Q}$ has been suggested by Bloch (for $r=-1$) and Beilinson($r\leq -1$). For a nontrivial $N$th root of unity $w$, the image by Beilinson regulator of the  element $\ell(w)\in K_{1-2r}(F)\otimes \mathbb{Q}$ is computed in \cite{beilinson1} and is given in $\oplus_{\sigma:F\rightarrow\mathbb{C}}(2\pi i)^{-r}\mathbb{R}$ by 
$$\rho^{r}_{F}(\ell(w))=(\mathrm{Li}_{1-r}(\sigma(w)))_{\sigma:F\rightarrow\mathbb{C}}$$
where for any complex number $s$ which verifies $\mid s\mid <1$, $\mathrm{Li}_{1-r}(s)$ is the polylogarithm function defined as
$$\mathrm{Li}_{1-r}(s)=\Sigma_{i=1}^{\infty}\frac{s^{i}}{i^{1-r}}.$$
which can be analytically extended to $\mathbb{C}-[1,\infty[$. If the first derivative $L^{'}(\chi, r)$ at $s=r$  is non zero, one can constuct an element in $K_{1-2r}(F)\otimes \mathbb{Z}[\chi]$ which maps to $L^{'}(\chi, r)$ by $\rho^{r}_{F}$, using Gross's formula (c.f. \cite{Burns-DeJeu-Gangl}, Proof of Thm. 3.1).\\
In \cite{Burns-DeJeu-Gangl}, the authors conjectured the existence of such elements for a wider choice of base number fields $k$. Unfortunately, the only available evidence for this conjecture is provided for $k=\mathbb{Q}$.\\
In the following work, we propose a proof of the conjecture when the base field is imaginary quadratic given that some conditions are fulfilled. Our argument is based on the Equivariant Tamagawa Number Conjecture, which is a generalization of both the analytic number formula, and the conjecutre of Birch and Swinnerton-Dyer. We adopt the formulation of Fontaine and Perrin-Riou (\cite{Fontaine}, \cite{fperrin-riou}) which generalizes to motives with coefficients in an algebra other than $\mathbb{Q}$. In this formulation, the ETNC links $K$-theory groups to motivic $L$-functions through isomorphisms which implicitly involve the Beilinson regulator. This, along with the recent proof of the ETNC over an imaginary quadratic number field for strictly negative integers given in \cite{Leung}, will enable us to prove the following results:
\begin{theorem}{(Thm. \ref{theorem4.4})}
Let $F/k$ be an abelian extension over the imaginary quadratic number field $k$. Suppose that $E$ is a number field containing all values of characters $\chi$ of $G:=\mathrm{Gal}(F/k)$ and that $r$ is a stricty negative integer. We make the identification $e_{\chi}E\cong E$.\\
Then the following statements are equivalent
\begin{itemize}
\item[1]Conjecture \ref{conjecture4.3.1}.1 holds for the pair $(E(r)_{F}, E[G])$.
\item[2]For each character $\chi$ of $G$ there exists an element $\tilde{\epsilon}_{\chi}(F)\in e_{\chi}(K_{1-2r}(F)\otimes E)$ which verifies
$$\rho_{F}^{r}(\tilde{\epsilon}_{\chi}(F)) = L^{'}(r,\chi^{-1})$$
where the Beilinson Regulator $\rho_{F}^{r}$ is defined here over $K_{1-2r}(F)\otimes E$ by extension of scalars.
\end{itemize}
\end{theorem}
\begin{theorem}{(Thm. \ref{theorem4.8})}
Let $F/k$ be a finite abelian extension of number fields with $k$ imaginary quadratic. Suppose that $p$ is a rational prime which does not divide $\# G$, that $E$ is a number field which contains all values of characters of $G$ and let $r$ denote a strictly negative integer. Let $S$ be a finite set of places of $k$ containing the infinite places, the $p$-places and the places which ramify in $F/k$ and $S_f$ the subset of finite places of $S$. Then the following statements are equivalent
\begin{itemize}
\item Conjecture \ref{conjecture4.3.4}.4 holds for the pair $(E(r)_{F}, O_{p}[G])$.
\item For all $p$-adic characters $\chi$ of $G$ one has
\small{
\begin{align*}
&\rho_{F}^{r}(e_{\chi}(K_{1-2r}(F)_{/tors}\otimes O_{p}))=\\ &\prod_{v\in S_{f}}(1-Nv^{-r}\chi^{-1}(v)).\mathrm{Fitt}_{O_{p}}(e_{\chi}(H^{0}(F, \mathbb{Q}_{p}/\mathbb{Z}_{p}(1-r))\otimes O_{p}))\mathrm{Fitt}^{-1}_{O_{p}}(e_{\chi}(K_{-2r}(O_{F, S})\otimes O_{p}))L^{'}(r,\chi^{-1})
\end{align*}
}
\end{itemize}
where $O_{F, S}$ refers to the ring of $S$-units of $F$, and $\rho_{F}^{r}$ denotes the extension of scalars of the Beilinson regulator over $O_p$. We also identified $e_{\chi}O_{p}[G]=e_{\chi}O_{p}$ with $O_{p}$.
\end{theorem}

\section{The conjecture}
\subsection{Beilinson regulator}
Let $r$ denote a negetive integer.\\
The Beilinson regulator defined over the $K$-group of the field of complex numbers is a map 
$$\rho^{r}_{\mathbb{C}}: K_{1-2r}(\mathbb{C})\rightarrow H^{1}_{D}(\mathrm{Spec}(\mathbb{C}),\mathbb{R}(1-r))\cong (2\pi i)^{-r}\mathbb{R}$$
where $H^{1}_{D}()$ is the first group of Deligne's Cohomology.\\
For a number field $F$ we compose with the map $$K_{1-2r}(F) \rightarrow \prod_{\sigma: F\rightarrow \mathbb{C}}K_{1-2r}(\mathbb{C}),$$
and obtain 
$$\rho^{r}_{F}: K_{1-2r}(F)\rightarrow \prod_{\sigma: F\rightarrow \mathbb{C}}(2\pi i)^{-r}\mathbb{R}\simeq X_{F}\otimes (2\pi i)^{-r}\mathbb{R},$$
where $X_{F}:=\mathbb{Z}[\Sigma_{F}]$ and $\Sigma_{F}$ is the set of embeddings of $F$ in $\mathbb{C}$. Let $\tau$ denote the complex conjugation automorphism $\in \mathrm{G}(\mathbb{C}/\mathbb{R})$. \\
The image of the Beilinson regulator over $K_{2n-1}(F)$ is invariant under the  action of $\tau$, hence we have 
$$\rho^{r}_{F}: K_{1-2r}(F)\rightarrow (\prod_{\sigma: F\rightarrow \mathbb{C}}(2\pi i)^{-r}\mathbb{R})^{+}.$$
Moreover, the image of $K_{1-2r}(F)$ by $\rho^{r}_{F}$ is a complete lattice of the $\mathbb{R}$-vector space $(\prod_{\sigma: F\rightarrow \mathbb{C}}(2\pi i)^{-r}\mathbb{R})^{+}$ 

\subsection{Burns-de Jeu-Gangl conjectural special elements for an abelian extension of number fields $F/k$}
Let $F/k$ be a finite abelian extension of number fields and $G=\mathrm{Gal}(F/k)$. Let $\chi$ be an irreducible (one dimensional) complex character of $G$, and $S$ be a finite set of places of $k$ containing the set of infinite places $S_\infty$.\\

We regard the set $\Sigma_F$ of embeddings $F\rightarrow \mathbb{C}$ as a left $G\times \mathrm{Gal}(\mathbb{C}/\mathbb{R})$-module by setting  $(g\times w)(\sigma)= w\circ \sigma\circ g^{-1}$, for all $g\in G$, $w\in \mathrm{Gal}(\mathbb{C}/\mathbb{R})$ and $\sigma\in \Sigma_F$.

Let $\mathbb{Z}[\chi]$ denote the ring of values of the character $\chi$, and suppose that $r$ is a strictly negative integer. We write
$$\rho^{r}_{F, \chi}: \mathbb{Z}[\chi]\otimes K_{1-2r}(F)\rightarrow \mathbb{Z}[\chi]\otimes (\prod_{\sigma :F\rightarrow\mathbb{C}}(2\pi i)^{-r}\mathbb{R})^{+}$$
for the map induced from $\rho^{r}_{F}$ by extension of scalars over $\mathbb{Z}[\chi]$.
We denote by $L_{S}^{'}(r,\chi)$ the first derivative of the $S$-truncated Artin $L$-function $L_{S}(s,\chi)$ of $\chi$ at $r$.
\begin{conjecture}[Conjecture \ref{conjecture2.1}]\label{conjecture2.1}(\cite{Burns-DeJeu-Gangl}, Conjecture $1.2$)
Assume that $r<0$ and that $L_{S}^{'}(r,\chi^{-1})\not = 0$. Then  for each $\sigma\in \Sigma_F$ there exists an element $\epsilon_{\sigma}(\chi, S)\in \mathbb{Z}[\chi]\otimes K_{1-2r}(F)$ such that
$$(2\pi i)^{r}\rho^{r}_{F, \chi}(\epsilon_{\sigma}(\chi, S))=w_{1-r}(F^{\mathrm{ker}(\chi)})L_{S}^{'}(r,\chi^{-1})(c^{\chi}_{\sigma,\sigma^{'}})_{\sigma^{'} : F\rightarrow \mathbb{C}}$$
where $F^{\mathrm{ker}}$ is the fixed field of $\chi$, $w_{1-r}(F^{\mathrm{ker}(\chi)}):=\mid H^{0}(\mathrm{Gal}(\overline{\mathbb{Q}}/F^{\mathrm{ker}(\chi)}), \mathbb{Q}/\mathbb{Z}(1-r))\mid$, and if we write  $\tau_{\sigma}$ for the generator of the decomposition subgroup of $G$ of the place of $F$ that corresponds to $\sigma$, the elements $c^{\chi}_{\sigma,\sigma^{'}}$ are given for each $\sigma^{'}\in \Sigma_F$ by
$$ c^{\chi}_{\sigma,\sigma^{'}} =  \left\{
\begin{array}{rl}
\chi^{-1}(g)+(-1)^{-r}\chi^{-1}(g\tau_{\sigma}) \;\;\;\;\;\;if\;\sigma(k)\subset \mathbb{R}\; and \; \sigma^{'}\circ g =\sigma\;for\;some\;g\in G, \\
\chi^{-1}(g)\;\;\;\;\;\;if\;\sigma(k)\not\subseteq \mathbb{R}\; and \; \sigma^{'}\circ g =\sigma\;for\;some\;g\in G, \\
(-1)^{-r}\chi^{-1}(g)\;\;\;if\;\sigma(k)\not\subseteq \mathbb{R}\; and \; \sigma^{'}\circ\tau\circ g =\sigma\;for\;some\;g\in G, \\
0\;\;\;\;\;\;\;\;\;\;\;\;\;\;\;\;\;\;\;\;\;\;\;\;\;\;\;\;\;\;\;\;\;\;\;\;\;\;\;\;\;\;\;\;\;\;\;\;\;\;\;\;\;\;\;\;\;\;\;\;\;\;\;\;\;\;\;\;\;\;\;\;\;\;\;\;otherwise.
\end{array}
\right.$$
\end{conjecture}
\begin{remarks}
\begin{itemize}
\item[1] By \cite{Burns-DeJeu-Gangl}, we know that Conjecture \ref{conjecture2.1} holds when $k=\mathbb{Q}$.
\item[2] If we write $\overline{\epsilon}_{\sigma}(\chi, S)$ to denote the image of  $\epsilon_{\sigma}(\chi, S)$ by the surjetive map 
$$K_{1-2r}(F)\otimes \mathbb{Z}[\chi] \rightarrow {(K_{1-2r}(F))}_{/tors}\otimes \mathbb{Z}[\chi]$$
where ${(K_{1-2r}(F))}_{/tors}$ is the torsion-free quotient of $K_{1-2r}(F)$ (i.e. the quotient of $K_{1-2r}(F)$ by its $\mathbb{Z}$-torsion submodule),
then, since $\rho^{r}_{F}$ is injective over ${(K_{1-2r}(F))}_{/tors}$ and $\mathbb{Z}[\chi]$ is a flat $\mathbb{Z}$-module, the element $\overline{\epsilon}_{\sigma}(\chi, S)$ is unique in ${(K_{1-2r}(F))}_{/tors}\otimes \mathbb{Z}[\chi]$.
\end{itemize}
\end{remarks}
\section{Case of an imaginary quadratic base number field}
In this section, we assume that the base field $k$ is imaginary quadratic. Recall that we suppose $r$ to be a strictly negative integer.

\noindent In this case $L_{S}(r, \chi)=0$  and $L_{S}^{'}(r, \chi)\not =0$  and the elements $c^{\chi}_{\sigma,\sigma^{'}}$ are given for each $\sigma^{'}\in \Sigma_F$ by
$$ c^{\chi}_{\sigma,\sigma^{'}} =  \left\{
\begin{array}{rl}

\chi^{-1}(g)\;\;\;\;\;\;if\; \sigma^{'}\circ g =\sigma\;for\;some\;g\in G, \\
(-1)^{-r}\chi^{-1}(g)\;\;\;if \; \sigma^{'}\circ\tau\circ g =\sigma\;for\;some\;g\in G.
\end{array}
\right.$$
We also have the following
\begin{lemma}\label{lemma3.1}
Let the number field $k$ be imaginary quadratic. Then there is a canonical isomorphism of $\mathbb{Z}[G]$-modules
$$ \iota :(\prod_{\sigma: F\rightarrow \mathbb{C}}(2\pi i)^{-r}\mathbb{R})^{+}\cong \mathbb{R}[G]$$
\end{lemma}
\begin{proof}
If $k$ is imaginary quadratic, then for each embedding $\hat{\sigma}\in \mathrm{Hom}(F, \mathbb{C})$, either $\hat{\sigma}$ or $\tau\hat{\sigma}$ identifies with an automorphism $\sigma\in G$.\\
The isomorphism $\iota$ is explicitly given by mapping,  for each such embedding $\hat{\sigma}\in \mathrm{Hom}(F, \mathbb{C})$, the element $(0,..,(2\pi i)^{-r}a_{\hat{\sigma}},..,(2\pi i)^{-r}a_{\tau\hat{\sigma}},..,0)$ (with $a_{\sigma}=\pm a_{\tau\sigma})$ to $a_{\hat{\sigma}}\sigma^{-1}$ if $\hat{\sigma}=\sigma$ and to $a_{\tau\hat{\sigma}}\sigma^{-1}$ if $\tau\hat{\sigma}=\sigma$.\hspace*{\fill}$\Box$
\end{proof}
Let $E$ be the  field extension generated over $\mathbb{Q}$ by the values of all characters $\chi\in \mathrm{Hom}(G,\mathbb{C})$ and $O$ be its ring of integers.\\
We extend the Beilinson regulator to the following by composing with the isomorphism $\iota$ and tensoring with $O$ 
$$\rho^{r}_{F}: K_{1-2r}(F)\otimes O \longrightarrow \mathbb{R}[G]\otimes O.$$
Explicitly if $a\in K_{1-2r}(F)$ and $x \in O$
$$\rho^{r}_{F}(a\otimes x) =\iota(\rho^{r}_{F}(a))\otimes x.$$
\begin{remark}
Let $d_{1-r}:=\mathrm{rk}_{\mathbb{Z}}(K_{1-2r}(F))$.
Since the image of $K_{1-2r}(F)$ by the Beilinson regulator is a $\mathbb{Z}$-lattice of $\mathbb{Z}$-rank $d_{1-r}$, and $O$ is a flat $\mathbb{Z}$ module, the image of $K_{1-2r}(F)\otimes O$ by $\rho^{r}_{F}$ is a free $O$-module of rank $d_{1-r}$ over $O$.
\end{remark}
We need the following proposition
\begin{proposition}\label{proposition3.2}
Suppose that the base field $k$ is imaginary quadratic and $\sigma\in G$. Let $S$ be any finite set of places of $k$ containing the set of infinite places. The image of the  element $\epsilon_{\sigma}(\chi, S)\in K_{1-2r}(F)\otimes O$ (where $\sigma$ is also viewed as an emebedding in $\Sigma_{F}$) by $\rho^{r}_{F}$ is given by
$$ \rho^{r}_{F}(\epsilon_{\sigma}(\chi, S)) =  
w_{1-r}(F^{\mathrm{ker}(\chi)})L_{S}^{'}(r,\chi^{-1})\chi^{-1}(\sigma)\mid{G}\mid e_{\chi},
$$
where $e_{\chi} := \frac{1}{\mid{G}\mid}\Sigma_{g\in G}\chi^{-1}(g)g$.
\end{proposition}
\begin{proof}
Since $k\not\subseteq \mathbb{R}$ Conjecture \ref{conjecture2.1} states that  
$$\rho^{r}_{F}(\epsilon_{\sigma}(\chi, S))= (2\pi i)^{-r}w_{1-r}(F^{\mathrm{ker}(\chi)})L_{S}^{'}(r,\chi^{-1})((-1)^{a}\chi^{-1}(h{\sigma^{'}}^{-1}\sigma))_{\sigma^{'}: F\rightarrow\mathbb{C}}$$
where $a=0$ and $h=\mathrm{Id}$ if $\sigma^{'}$ identifies with an automorphism of $G$ and $a=-r$ and $h=\tau$ otherwise.\\
If $r$ is even, $(-1)^{a}=1$ and the extension of scalars of the isomorphism $\iota$ over $O$  sends the element $((2\pi i)^{-r}\chi^{-1}(h{\sigma^{'}}^{-1}))_{\sigma^{'}: F\rightarrow\mathbb{C}}$ to $\Sigma_{\sigma\in G}\chi^{-1}({\sigma}^{-1}){\sigma}^{-1}=\mid G\mid e_{\chi}$.\\
The case of $r$ odd is proved in a similar way.
\hspace*{\fill}$\Box$
\end{proof}
Let $p$ be a rational prime, and let $O_p:=O\otimes \mathbb{Z}_p$. We want to formulate a $p$-adic analog of conjecture \ref{conjecture2.1}.\\
First, the isomorphism $\iota$ in Lemma \ref{lemma3.1}, is also an isomorphism of $\mathbb{R}$-vector spaces. Since the image by the Beilinson regulator of $K_{1-2r}(F)$ is a complete $\mathbb{Z}$-lattice in the $\mathbb{R}$-vector space $(\prod_{\sigma:F\rightarrow\mathbb{C}}(2\pi i)^{-r}\mathbb{R})^{+}$, it is also  
  a complete $\mathbb{Z}$-lattice $\mathcal{T}$ in $\mathbb{R}[G]$ once we compose with the isomorphism $\iota$. We can, then,  tensor $\mathcal{T}$ by the ring $O_p$, and define a $p$-adic regulator by extension of scalars over $O_p$
$$\rho^{r}_{F, p}: K_{1-2r}(F)\otimes_{\mathbb{Z}} O_p\rightarrow \mathcal{T}\otimes_{\mathbb{Z}} O_p$$
Note that $\mathcal{T}\otimes_{\mathbb{Z}} O_p\simeq O_p^{\mathrm{rk}_{\mathbb{Z}}K_{1-2r}(F)}=O_p^{\mid G\mid}$.\\

By Proposition \ref{proposition3.2}, we can reformulate conjecture \ref{conjecture2.1} for the extension $F/k$ as follows
\begin{conjecture}[Conjecture \ref{conjecture3.1} ($p$-adic reformulation of conjecture \ref{conjecture2.1})]\label{conjecture3.1}
Let $k$ be an imaginary quadratic number field, and $F/k$ a finite abelian extension.
Suppose that $r$ is a strictly negative integer  and that $p$ is a rational prime. Let $S$ be a finite set of places of $k$ containing the set of infinite places.\\
Then  for each $\sigma\in G:=\mathrm{Gal}(F/k)$ and each one dimensional character $\chi$ of $G$ the following holds
\begin{itemize}
\item[1] The element $w_{1-r}(F^{\mathrm{ker}(\chi)})L_{S}^{'}(r,\chi^{-1})\chi^{-1}(\sigma)\mid{G}\mid e_{\chi}$ of the $\mathbb{R}$-vector $\mathbb{R}[G]=\mathcal{T}\otimes \mathbb{R}$ belongs to the $E$-vector space $\mathcal{T}\otimes E$ (recall that $E$ is the field generated over $\mathbb{Q}$ by all values of characters of $G$).
\item[2] Suppose that the previous condition is fulfilled. Then, there exists an element $\epsilon_{\sigma}(\chi, S, p)\in K_{1-2r}(F)\otimes O_p$ such that
$$ \rho^{r}_{F, p}(\epsilon_{\sigma}(\chi, S, p)) =  
w_{1-r}(F^{\mathrm{ker}(\chi)})L_{S}^{'}(r,\chi^{-1})\chi^{-1}(\sigma)\mid{G}\mid e_{\chi},
$$
where $e_{\chi} := \frac{1}{\mid{G}\mid}\Sigma_{g\in G}\chi^{-1}(g)g$.
\end{itemize}
\end{conjecture}
\begin{remarks}
\item[1] Since $\rho^{r}_{F, p}(K_{1-2r}(F)\otimes O_p)=\mathcal{T}\otimes O_p$,
the second statement in conjecture \ref{conjecture3.1} is equivalent to the assumption that
$$w_{1-r}(F^{\mathrm{ker}(\chi)})L_{S}^{'}(r,\chi^{-1})\chi^{-1}(\sigma)\mid{G}\mid e_{\chi}\in \mathcal{T}\otimes O_p$$
\item[2] To prove that Conjecture \ref{conjecture3.1} holds for an abelian extension $F/k$ and a character $\chi$ of $G$, it suffices to prove the existence of the element $\epsilon_{\sigma}(\chi, S, p)$ for one given choice of automorphism $\sigma\in G$. In fact
if $\sigma^{'}\in G$ is another automorphism, then by Proposition 3.2 one has (modulo tosion)
$$\epsilon_{\sigma^{'}}(\chi, S, p)=\chi^{-1}(\sigma^{'}\sigma^{-1})\epsilon_{\sigma}(\chi, S, p).$$
\item[3]  Suppose that $p$ doesn't divide $\mid G\mid$.
The image $\overline{\epsilon}_{\sigma}(\chi, S, p)$ of the element $\epsilon_{\sigma}(\chi, S, p)$ in ${(K_{1-2r}(F))}_{/tors}\otimes O_p$ lies  in $e_\chi( {(K_{1-2r}(F))}_{/tors}\otimes O_p)$.\\This is because $\rho^{r}_{F, p}$ is injective over $(K_{1-2r}(F))_{/tors}\otimes O_p$ (since $O_p$ is a flat $\mathbb{Z}$module), and
\begin{align*}
\rho^{r}_{F, p}(\epsilon_{\sigma}(\chi, S, p))&=e_{\chi}\rho^{r}_{F, p}(\epsilon_{\sigma}(\chi, S, p))\\
&=\rho^{r}_{F, p}(e_{\chi}\epsilon_{\sigma}(\chi, S, p))
\end{align*}
\end{remarks}

Suppose, as before, that $r<0$, and that $S$ is a finite set of places of the totally imaginary quadratic field $k$ containing the infinite places.  Fix a character $\chi$ of $G$. Let $\mathcal{E}$ be a number field which contains all values of $\chi$, and $\mathcal{O}$ the ring of integers of $\mathcal{E}$. We will say that Conjecture \ref{conjecture2.1} holds for the set of data $(F/k, r, S, \chi, \mathcal{O})$, if the elements $\epsilon_{\sigma}(\chi, S)$ of Conjecture \ref{conjecture2.1} exist within $K_{1-2r}\otimes \mathcal{O}$ for all $\sigma\in G$. For example, if Conjecture \ref{conjecture2.1} holds exactly as stated above for all $\sigma\in G$, we will say that it holds for the set of data $(F/k, r, S, \chi, \mathbb{Z}[\chi])$.

If $\mathcal{E}$ contains values of all characters of $G$, and if Conjecture \ref{conjecture2.1} holds for all sets of data $(F/k, r, S, \chi, \mathcal{O})$ for all characters $\chi$ of $G$, then we simply say that Conjecture 2.1 holds for the set $(F/k, r, S, \mathcal{O})$.\\
Similarly, we will say that conjecture \ref{conjecture3.1} holds for the set of data $(F/k, r, S, \mathcal{O}_p)$, if it does hold for the  rational prime $p$ and  all characters $\chi$ of $G$.

\begin{proposition}\label{proposition3.3}
If conjecture \ref{conjecture2.1} holds for the set of data $(F/k, r, S, \chi, \mathcal{O})$, then it holds for $(F/k, r, S, \chi, \mathcal{O}^{+}[\chi])$, where $\mathcal{O}^{+}$ is the ring of integers of the maximal real subfield $\mathcal{E}^{+}$ of $\mathcal{E}$. 
\end{proposition}
\begin{proof}
Let $\sigma$ be an automorphism of $G$ and let $\chi$ be a one dimensional character of $G$. Suppose there exists an element $\epsilon_\sigma(\chi, S)^{'}\in K_{1-2r}(F)\otimes \mathcal{O}$ such that
$$\rho^{r}_{F}(\epsilon_{\sigma}(\chi, S)^{'}) =  
w_{1-r}(F^{\mathrm{ker}(\chi)})L_{S}^{'}(r,\chi^{-1})\chi^{-1}(\sigma)\mid{G}\mid e_{\chi}.$$
This means that
$$x:= w_{1-r}(F^{\mathrm{ker}(\chi)})L_{S}^{'}(r,\chi^{-1})\chi^{-1}(\sigma)\mid{G}\mid e_{\chi}\in \mathcal{T}\otimes \mathcal{O}$$
Let $g:=\mid G\mid$ and let $(a_i)_{i=1,...,g}$ be a $\mathbb{Z}$-basis of $\mathcal{T}$. Since $x\in \mathcal{T}\otimes \mathcal{O}$, there exist elements $x_1,...,x_g$ in $\mathcal{O}$ such that
\begin{equation}\label{eq1}
x=\Sigma_{i=1}^{g}x_i a_i
\end{equation}
Consider now the complete $\mathbb{Z}$-lattice $\mathcal{T}^{'}$ of the $\mathbb{R}$-vector space $\mathbb{R}[G]$ generated by the basis $(b_h:=L_{S}^{'}(r,\chi^{-1})h)_{h\in G}$ (it is implied that $L_{S}^{'}(r,\chi^{-1})\not =0$).
The element $x$ is written in the basis $(b_h)_{h\in G}$ as a sum
$$x=\Sigma_{h\in G}x_{h}^{'} b_h,$$
where, for each $h\in G$, $x_{h}^{'}=w_{1-r}(F^{\mathrm{ker}(\chi)})\chi^{-1}(\sigma h)\in \mathbb{Z}[\chi]$.
Since $\mathbb{R}\mathcal{T}=\mathbb{R}\mathcal{T}^{'}(=\mathbb{R}[G])$, by change of basis we get again
\begin{equation}\label{eq2}
x=\Sigma_{i=1}^{g}y_i a_i
\end{equation}
with $y_i \in \mathbb{R}(\chi)$, for each $i=1,...,g$. We compare now (\ref{eq1}) and (\ref{eq2}). For this we should first study the linear dependency of the family of vectors $(a_i)_{i=1,...,g}$ over $\mathbb{C}$.\\
The $\mathbb{Z}$-basis $(a_i)_{i=1,...,g}$ is also a basis of the vector space $\mathbb{R}T=\mathbb{R}[G]$. If, the family of vectors $(a_i)_{i=1,...,g}$ 
is linearly dependent over the field of complex numbers, then the $\mathbb{C}$-vector space $\Sigma_{i=1}^{g}\mathbb{C}a_i$ has necessarily dimension strictly less than $g$. However, since $\Sigma_{i=1}^{g}\mathbb{R}a_i=\mathbb{R}[G]$, we get 
\begin{align*}
\Sigma_{i=1}^{g}\mathbb{C}a_i&=\mathbb{C}[G]
\end{align*}
which shows indeeed that $(a_i)_{i=1,...,g}$ are linearly independent over $\mathbb{C}$. Therefore for each $i=1,...,g$, we have $x_i =y_i \in \mathcal{O}\cap \mathbb{R}(\chi)\subseteq \mathcal{O}^{+}[\chi]$.\\
\hspace*{\fill}$\Box$
\end{proof}
 
\begin{proposition}\label{proposition3.4}
If Conjecture \ref{conjecture3.1} holds for all sets of data $(F/k, r, S, \mathcal{O}_p)$ for all $p$-adic integers $p$, then Conjecture \ref{conjecture2.1} holds for $(F/k, r, S, \mathcal{O})$.\\
If Conjecture \ref{conjecture3.1} holds for  $(F/k, r, S, \mathcal{O}_p)$ for all but a finite set of rational primes $\{p_1,...,p_n\}$, then Conjecure \ref{conjecture2.1} holds for $(F/k, r, S, \mathcal{O}[1/m])$, where $m=\prod_{i=1}^{n}p_i\in\mathbb{Z}$.
\end{proposition}
\begin{proof}
Fix a choice of an abelian extension $F/k$ with $k$ imaginary quadratic, an integer $r<0$ and a finite set $S$ of places of $k$ containing the set of infinite places.\\
Let $\sigma\in G$ and $\chi$ be a character of $G$. Suppose that conjecture \ref{conjecture3.1} holds for $(F/k, r, S, \mathcal{O}_p)$ for all primes $p$. In particular, for each rational prime $p$, there exists an element $\epsilon_{\sigma}(\chi, S, p)\in K_{1-2r}\otimes \mathcal{O}_p$ which verifies the two statements of conjecture \ref{conjecture3.1}.\\
The first statement of Conjecture \ref{conjecture3.1} ensures that  
$$x:=w_{1-r}(F^{\mathrm{ker}(\chi)})L_{S}^{'}(r,\chi^{-1})\chi^{-1}(\sigma)\mid{G}\mid e_{\chi}\in \mathcal{T}\otimes \mathcal{E}$$
(recall that $\mathcal{E}$ is the field of fractions of $\mathcal{O}$).
Let $g:=\mid G\mid$ and  $(a_i)_{i=1,...,g}$ be a basis of the $\mathbb{Z}$-lattice $\mathcal{T}$.
By the above, there exist coefficients $x_i\in \mathcal{E}$, for $i=1,...,g$ such that
$$x=\Sigma_{i=1}^{g}x_i a_i$$
Yet, the second  statement of Conjecture \ref{conjecture3.1} shows that
$$x\in \mathcal{T}\otimes \mathcal{O}_p \mathrm{\;for\;all\;rational\;primes\;}p$$
which means that $x_i\in \mathcal{O}_p$ for all primes $p$. Fix an integer $i\in\{1,...,g\}$. Since $\mathcal{E}$ is the field of fractions of $\mathcal{O}$, there exist $b_i$ and  $c_i$ relatively prime in $\mathcal{O}$, such that $x_i =b_i/c_i$. As, for each $p$, $x_i\in \mathcal{O}_p$, the denominator $c_i$ is necessarily not divisible by $p$. Hence, $c_i$ is not divisible by all primes $p$ and is therefore a unit in $\mathcal{O}$. which gives  $x_i\in \mathcal{O}$ and the first claim ensues.\\
Similarly, if $x_i\in \mathcal{O}_p$ for all but a finite set of primes $\{p_1,...,p_n\}$, then $c_i$ is only divisible by primes $p\in \{p_1,...,p_n\}$. Thus 
$x_i\in \mathcal{O}[\frac{1}{p_1p_2...p_n}]$.\hspace*{\fill}$\Box$

\end{proof}

\section{Relationship with the Equivariant Tamagawa Number Conjecture}
\subsubsection*{The functor $\mathrm{Det}$}
Let $R$ be a ring and $P$ a finitely generated projective $R$-module. For every prime ideal $\mathfrak{p}\in \mathrm{Spec}(R)$, the localization $P_{\mathfrak{p}}$ is a finitely generated free $R_\mathfrak{p}$-module (since $R_\mathfrak{p}$ is local).
\begin{definition}
The rank of a finitely generated projective $R$-module at a
prime ideal $\mathfrak{p}$ of $R$ is the integer $\mathrm{rk}_{\mathfrak{p}}P = \mathrm{rk}_{R_\mathfrak{p}} P_\mathfrak{p}$. 
\end{definition}
Note that $\mathrm{rk}_{\mathfrak{p}}P$ is well defined (unique) since local fields have IBN (invariant basis number).\\
One can also verify that alternatively
 $$\mathrm{rk}_{\mathfrak{p}}P=\mathrm{rk}_{\kappa_{\mathfrak{p}}}(P\otimes_{R} \kappa_{\mathfrak{p}})$$
where $\kappa_{\mathfrak{p}}=R_{\mathfrak{p}}/\mathfrak{p}R_{\mathfrak{p}}$ is the residue field at $\mathfrak{p}$.\\
For a fixed finitely generated projective $R$-module $P$, one can define a function
$\mathrm{Spec}(R)\rightarrow \mathbb{Z}$, which maps any prime $\mathfrak{p}\in \mathrm{Spec}(R)$ to $\mathrm{rk}_{\mathfrak{p}}P$. We recall the following (\cite{Bourbaki}, Chapter II, §5.3, Theorem 1)\\
\textbf{Property}
\textit{The function $\mathrm{rk}_{P} : \mathrm{Spec}(R) \rightarrow \mathbb{Z}$ is  locally constant (and hence continuous with
respect to the Zariski topology on $\mathrm{Spec}(R)$ and the discrete topology on $\mathbb{Z}$) and bounded.}\\

\noindent We say that the finitely generated projective $R$-module $P$ has constant rank over $R$, if the associated function $\mathrm{rk}_{P} : \mathrm{Spec}(R) \rightarrow \mathbb{Z}$ is constant. We write in this case $\mathrm{rk}_{R}(P)=\mathrm{rk}_{P}(\mathrm{Spec}(R))$.\\

\noindent\textbf{Corollary}
\textit{If $\mathrm{Spec}(R)$ is connected (which is equivalent to $R$ having no nontrivial idempotents), then every finitely generated
projective module over $R$ has constant rank.}
\\

\noindent\textbf{Example}
\textit{If $R$ is an integral domain then every finitely generated
projective $R$-module has constant rank.}\\
\begin{definition}
Let $A^{\bullet}$ denote a complex of $R$-modules. The complex $A^{\bullet}$ is said to be perfect if there exists  a bounded complex $P^{\bullet}$ of finitely generated  projective $R$-modules and a morphism of complexes $\pi : P^{\bullet}\rightarrow A^{\bullet}$ which induces isomorphisms for all  $i$
$$H^{i}(P^{\bullet})\xrightarrow{\;\sim\;}   H^{i}(A^{\bullet}).$$
(In this case we say that $\pi$ is a quasi-isomorphism).
\end{definition}
\textbf{Remark}\\
\textit{An $R$-module is said to be perfect if it is perfect when considered as the complex $0\rightarrow A\rightarrow 0$.}\\

\noindent We denote by $\mathcal{D}^{\mathrm{p}}(R)$ the category of perfect complexes of $R$-modules.\\
Let $P$ be a finitely generated projective $R$-module. Let $R=\oplus_{i=1}^{s}R_{i}$ be a decomposition of $R$ as a direct finite sum of rings $R_{i}$ with connected spectra. For each $i$ one has  $R_{i}=e_{i}.R$, where $e_{i}$ is an indecomposable idempotent (e.g. \cite{Popescu}, §1.3). We have a direct sum decomposition $P=\oplus_{i=1}^{s}P_{i}$, where $P_{i}=e_{i}P$ is a projective $R_{i}$-module of constant  rank $r_{i}$.\\
The Knudsen-Mumford determinant functor \cite{KM} of $P$ over $R$, is then the graded invertible (projective of constant rank 1) $R$-module
$$\mathrm{Det}_{R}P:=\oplus_{i=1}^{s}\bigwedge^{r_{i}}_{R_{i}}P_{i}$$
Let $A^{\bullet}$ be an object of $\mathcal{D}^{\mathrm{p}}(R)$ and suppose that $A^{\bullet}$ is quasi-isomorphic to the bounded complex of finitely generated projective $R$-modules $P^{\bullet}$. The Knudsen-Mumford determinant functor  of $A^{\bullet}$ over $R$ depends only on the quasi-isomorphism class of $A^{\bullet}$ and is defined as 
$$\mathrm{Det}_{R}A^{\bullet}=\mathrm{Det}_{R}P^{\bullet}:=\otimes_{i\in\mathbb{Z}}\mathrm{Det}_{R}^{(-1)^{i}}P^{i}.$$
\begin{remark}
We used in the defintion of $\mathrm{Det}_{R}A^{\bullet}$ the parity convention of \cite{KM} rather than that of \cite{burnsGreither}. However, this will not affect our results and the formulation of ETNC with this convention is equivalent to the formulation given in \cite{burnsGreither}.
\end{remark}
We recall some properties of the Knudsen-Mumford determinant functor (e.g. \cite{burnsGreither}, §2)\\

\noindent\textbf{Properties of the functor $\mathrm{Det}$}\\
\begin{itemize}
\item[1.]$\mathrm{Det}^{-1}_{R}A^{\bullet}=\mathrm{Hom}_{R}(\mathrm{Det}_{R}A^{\bullet}, R)$.
\item[2.]If the cohomology groups $H^{i}(A^{\bullet})$ are all perfect then
$$\mathrm{Det}_{R}A^{\bullet}=\otimes_{i\in\mathbb{Z}}\mathrm{Det}_{R}^{(-1)^{i}}H^{i}(A^{\bullet})$$
\item[3.]If $C_{1}^{\bullet}\rightarrow C_{2}^{\bullet}\rightarrow C_{3}^{\bullet}$ is a distinguished triangle (e.g. \cite{nekovar}, 1.1.3) of perfect complexes of $R$-modules then
$$\mathrm{Det}_{R}C_{2}^{\bullet}\simeq \mathrm{Det}_{R}C_{1}^{\bullet}\otimes \mathrm{Det}_{R}C_{3}^{\bullet}$$
\item[4.]If $A$ is a finitely generated torsion $R$-module which has a projective dimension at most 1, then $A$ is perfect and
$$\mathrm{Det}_{R}A=(\mathrm{Fitt}_{R}A)^{-1}.$$
\end{itemize}

\subsubsection*{The fundamental line}
Let $F/k$ be an abelian extension of number fields, $E$ a number field. \textit{We suppose in the rest of the article that $r$ is a strictly negative integer}. We set
$$E(r)_{F}:=h^{0}(\mathrm{Spec}(F))_{E}(r).$$
for the motive of $F$ over $k$ with coefficients in $E$ and twist $r$.
Since $r<0$ we set
$$Y_{r}(F)=\prod_{\sigma\in \Sigma_{F}}(2\pi i)^{-r}\mathbb{Z}.$$
where $\Sigma_{F}$ denotes again the set of embeddings of $F$ in the field of complex numbers.

For a motive $M$ of $F$ over $k$ with coefficients in $E$ one has a dual motive $M^{\vee}$ with dual realisations. In general, if $X$ is a smooth projective variety of dimenion $d$, and $M=h^{i}(X)(j)$, then $M^{\vee}$ identifies with $h^{2d-i}(X)(d-j)$ (e.g. \cite{FlachSurvey}, Part 1, §2).\\ For $M=E(r)_{F}$, $r<0$, we define the fundamental line in terms of motivic cohomology to be 
$$\Xi(M):=\mathrm{Det}_{E[G]}(H^{0}_{\mathcal{M}}(M))\otimes \mathrm{Det}_{E[G]}((H^{1}_{\mathcal{M}}(M^{\vee}(1))^{\vee})\otimes \mathrm{Det}^{-1}_{E[G]}(M_{B}^{+})$$
where $(H^{1}_{\mathcal{M}}(M^{\vee}(1))^{\vee}=\mathrm{Hom}_{E}(H^{1}_{\mathcal{M}}(M^{\vee}(1),E)$, and $M_B$ denotes the Betti realization of $M$.\\
For $M=E(r)_{F}$, one has $H^{0}_{\mathcal{M}}(E(r)_{F})=0$ ($r<0$), $H^{1}_{\mathcal{M}}(E(r)_{F}^{\vee}(1))=K_{1-2r}(O_{F})\otimes E$ and $(E(r)_{F})_{B}\cong Y_{r}(F)^{\vee}\otimes E $.\\

\noindent Let the map $x\mapsto x^{\#}$ denote the $\mathbb{Z}$-linear involution of the group ring $\mathbb{Z}[G]$ which satisfies $g^{\#}=g^{-1}$ for each $g\in G$. If $X$ is any (complex of) $\mathbb{Z}[G]$-module(s), then we write $X^{\#}$ for the scalar extension of $X$ with respect to the morphism $x\mapsto x^{\#}$.\\

\noindent If $X$ is an object of the category $\mathcal{D}^{p}(E[G])$ of perfect complexes of $E[G]$-modules we denote by $X^{\vee}$ the complex $X^{\vee}=\mathrm{RHom}_{E}(X,E)$ which is endowed with the contragredient $G$-action and is also an object of $\mathcal{D}^{p}(E[G])$.\\

Recall that we have a canonical isomorphism  $M^{\vee}\cong \mathrm{Hom}_{E[G]}(M, E[G])^{\#}$ for any $E[G]$-module $M$ (e.g. \cite{Popescu}, Remark 1.1.1), which induces for each object $X$ in the category of perfect complexes of $E[G]$-modules the following canonical isomorphism (e.g. \cite{burnsGreither}, §2 (after Lemma 2.1))
$$\mathrm{Det}_{E[G]}X^{\vee}\cong \mathrm{Det}_{E[G]}^{-1}X^{\#}$$
In fact, since the functor $\mathrm{Det}_{E[G]}$ only depends on the quasi-isomorphism class of $X$, it suffices to show the result for a bounded complex of finitely generated projective $E[G]$-modules $P^{\bullet}$.\\
One has 
$$\mathrm{Det}_{E[G]}(P^{\bullet})^{\vee}:=\mathrm{Det}_{E[G]}\mathrm{RHom}_{E}(P^{\bullet}, E)$$
Yet, $\mathrm{RHom}_{E}(P^{\bullet}, E)$ is quasi-iomorphic to the bounded complex of finitely generated projective $E[G]$-modules whose $(-n)$-th entry is $\mathrm{Hom}_{E}(P^{n}, E)$. Hence
\begin{align*}
\mathrm{Det}_{E[G]}(P^{\bullet})^{\vee}&=\otimes_{i\in\mathbb{Z}}\mathrm{Det}^{(-1)^{i}}_{E[G]}(P^{i})^{\vee}\\
&\cong\otimes_{i\in\mathbb{Z}}\mathrm{Det}^{(-1)^{i}}_{E[G]}\mathrm{Hom}_{E[G]}(P^{i},E[G])^{\#}
\end{align*}
Clearly, we have $\mathrm{Hom}_{E[G]}(P^{i},E[G])^{\#}=\mathrm{Hom}_{E[G]}((P^{i})^{\#},E[G])$. Thus
\begin{align*}
\mathrm{Det}_{E[G]}(P^{\bullet})^{\vee}&\cong\otimes_{i\in\mathbb{Z}}\mathrm{Det}^{(-1)^{i}}_{E[G]}\mathrm{Hom}_{E[G]}((P^{i})^{\#},E[G])\\
&\cong \otimes_{i\in\mathbb{Z}}\mathrm{Det}^{(-1)^{i+1}}_{E[G]}(P^{i})^{\#}\;\mathrm{(use\; e.g.\;}\cite{Popescu}\mathrm{\;,Remark\;1.1.2)}\\
&=\mathrm{Det}^{-1}_{E[G]}(P^{\bullet})^{\#}
\end{align*}
We apply this result to the fundamental line and get the following
$$\Xi(E(r)_{F})= \mathrm{Det}^{-1}_{E[G]}(K_{1-2r}(F)\otimes E)^{\#}\otimes \mathrm{Det}_{E[G]}(Y_{r}(F)^{+}\otimes E)^{\#},$$
This is the definition of the fundamental line as given in (\cite{burnsGreither}, §3.1), but with inverted signs of $\mathrm{Det}$ since we adopted here the parity convention of \cite{KM}.
 \subsection{Cohomological data}
\subsubsection{\small{$f$-cohomology}}
For any commutative ring $Z$ and any étale sheaf $\mathcal{F}$, we abbreviate in the sequel $R\Gamma(\mathrm{Spec}(Z)_{ét}, \mathcal{F})$ and $H^{i}(R\Gamma(\mathrm{Spec}(Z)_{ét}, \mathcal{F}))$ to $R\Gamma(Z, \mathcal{F})$ and $H^{i}(Z, \mathcal{F})$ respectively.\\
Let $p$ be a rational prime and let us denote by $M_{p}:=H^{0}_{ét}(\mathrm{Spec}(F)\times_{k}\overline{k}, E_p(r))$ (with $E_p:=E\otimes \mathbb{Q}_p$) the étale realization of the motive $E(r)_{F}$. Following Fontaine \cite{Fontaine}, for every finite place $v$ of $k$ define the local unramified cohomology of $M_p$ to be the complex
$$ R\Gamma_{f}(k_{v}, M_{p})=
\left\{
\begin{array}{rl}
M_{p}^{I_{v}}\stackrel{1-\mathrm{Frob}_{v}^{-1}}{\rightarrow} M_{p}^{I_{v}}\;\;\;v\nmid p \\
(B_{cris}\otimes M_{p})^{G_{k_v}}\stackrel{1-\mathrm{Frob_{v}^{-1}}}{\rightarrow}(B_{cris}\otimes M_{p})^{G_{k_v}}\oplus (B_{\mathrm{dR}}/\mathrm{Fil}^{0}\otimes M_{p})^{G_{k_{v}}}\;\;\;v\mid p
\end{array}
\right.
$$
where $I_v$ is the inertia group at $v$, and $B_{cris}\subset B_{\mathrm{dR}}$ are Fontaine's rings of $p$-adic periods (cf. \cite{Fontaine1}). We also write $(\mathrm{Fil}^{i})_{i\in\mathbb{Z}}=(\mathrm{Fil}^{i}(B_{\mathrm{dR}}))_{i\in\mathbb{Z}}$ for the decreasing filtration associated to $B_{\mathrm{dR}}$.\\
Note that the first cohomology group $H^{1}_{f}(k_{v}, M_{p})$ of the complex $\mathrm{R\Gamma }_{f}(k_{v}, M_{p})$ is Bloch-Kato's local condition (\cite{BK}, §3, (3.8.2)) or Mazur-Rubin's finite condition (\cite{MazurRubin}, Def. 1.1.6) defined by
$$ H^{1}_{f}(k_{v}, M_{p})=
\left\{
\begin{array}{rl}
\mathrm{ker}(H^{1}(k_{v}, M_{p})\rightarrow H^{1}(I_{v}, M_{p}))\;\;\;v\nmid p\\
\mathrm{ker}(H^{1}(k_{v}, M_{p})\rightarrow H^{1}(k_{v}, M_{p}\otimes B_{cris}))\;\;\;v\mid p
\end{array}
\right.
$$
It is also worth noting that the tangent space $t_{v}(M):=(B_{\mathrm{dR}}/\mathrm{Fil}^{0}\otimes M_{p})^{G_{k_{v}}}$ is trivial for our motive. In fact  by (\cite{burnsFlach1}, §1.4) one has an isomorphism
$$\oplus_{v\in S_{p}}t_{v}(M)\cong (H_{dR}(M)/\mathrm{Fil}^{0}(H_{dR}(M))\otimes E_{p}$$
where $E_{p}:=E\otimes \mathbb{Q}_{p}$, and $S_p$ is the set of places of $k$ above $p$. However for $M= E(r)_{F}$, we have $H_{dR}(M)=\mathrm{Fil}^{0}(H_{dR}(M)= F$ (e.g. \cite{burnsFlach1}, §1.3) or \cite{burnsFlach1}, §1.1, p. 70).\
\\
If the place $v$ is archimedean, we define the complex $R\Gamma_{f}(k_{v}, M_{p})$ of $f$-cohomology as
$$R\Gamma_{f}(k_{v}, M_{p}):= R\Gamma(k_{v}, M_{p})\;\;\;\;v\mid\infty$$
We Define also the complex
$$R\Gamma_{/f}(k_{v},M_{p}):= \mathrm{Cone}(R\Gamma_{f}(k_{v},M_{p})\rightarrow R\Gamma(k_{v},M_{p}))$$
where Cone() is the cone functor (e.g. \cite{nekovar}, (1.1.2)).\\
Note that when $v\mid \infty$, the complex $R\Gamma_{/f}(k_{v},M_{p})$ is acyclic.\\

For any finite set $S$ of places of $k$ containing the set of $p$-places, and the set $S_\infty$ of infinite places, we define the global unramified cohomology as 
$$R\Gamma_{f}(O_{k,S} ,M_{p}):= \mathrm{Cone}(R\Gamma(O_{k,S},M_{p})\rightarrow \bigoplus_{v\in S}R\Gamma_{/f}(k_{v},M_{p}))[-1]$$
where $O_{k,S}$ denotes the ring of $S$-integers of $k$.
\subsubsection{\small{Compact Support cohomology}}
Let $T_p$ be any Galois stable lattice of $M_p$, and $S$ a finite set of places of $k$ containing the set $S_p$ of places of $k$ above $p$ and the set $S_\infty$ of infinite places of $k$.
The cohomology with compact support (e.g. \cite{nekovar}, 5.3.1) is defined so as to lie in a canonical distinguished triangle
$$R\Gamma_{c}(O_{k,S}, T_p)\rightarrow R\Gamma(O_{k,S}, T_p)\rightarrow \oplus_{v\in S}R\Gamma(k_{v}, T_p),$$
where, we write $R\Gamma(O_{k,S}, T_p)$ for the complex of étale cohomology $R\Gamma_{et}(\mathrm{Spec}O_{k,S}, T_p)$.
\begin{remark}
Let $E_{p}:=E\otimes \mathbb{Q}_{p}$ and $O_{p}:=O\otimes \mathbb{Z}_{p}$, where $O$ denotes the ring of integers of $E$. Let $S$ be a finite et of place of $k$ containing the set $S_p$ of $p$-places of $k$.
\begin{itemize}
\item[1] The complex $R\Gamma_{f}(k_{v}, M_{p})$ is perfect over $E_{p}[G]$ (e.g. \cite{Burns}, Lemma 12.2.1).
\item[2] Let $T_{p}$ denote any Glois-stable lattice inside $M_p$. The complexes $R\Gamma_{c}(O_{k,S}, T_p)$ is perfect over $O_{p}[G]$ (e.g. \cite{burnsFlach1}, Proposition 1.20).
\item[3] Since $E_{p}[G]$ is semisimple, by (\cite{nekovar}, 4.2.8), a complexe of $E_{p}[G]$-modules is perfect if and only if it has bounded finitely generated cohomology. If $k$ is totally imaginary, the complexes $R\Gamma(O_{k,S}, M_p)$ and $R\Gamma(k_{v}, M_{p})$ are then perfect. 
\item[4] The complexes $R\Gamma_{/f}(k_{v}, M_{p})$ and $R\Gamma_{f}(O_{k,S}, M_p)$ are also perfect over $E_{p}[G]$. This ensues from their respective definitions and the points $1$ and $3$ made above.
\item[5] If $p$ doesn't divide $\#G$, the previous observations can be also used to show for example that the complexe $R\Gamma(O_{k,S}, T_p)$ is perfect over $O_{p}[G]$.
\end{itemize}
\end{remark}
\subsection{The isomorphisms $\vartheta^{r}_{F, S_p}$ and $\vartheta^{r}_{F, \infty}$}
\subsubsection{The isomorphisms $\vartheta^{r}_{F, S_p}$}
\begin{lemma}\label{lemma4.1}
For every sequence of morphisms of complexes
$$X\stackrel{f}\longrightarrow Y \stackrel{g}\longrightarrow Z$$
one has a distinguished triangle
$$\mathrm{Cone}(f)\longrightarrow \mathrm{Cone}(g\circ f)\longrightarrow \mathrm{Cone}(g).$$
\end{lemma}
\begin{proof}
e.g. \cite{milne}, Chap. II, Proposition 0.10.$\Box$
\end{proof}
Let $S_p$ denote the set of primes of $k$ above $p$. In the sequel, we  always abbreviate $R\Gamma_{?}(O_{k, S_{p}\cup S_\infty}, ..)$ as $R\Gamma_{?}(O_{k, S_p}, ..)$, and $H^{i}_{?}(O_{k, S_{p}\cup S_\infty}, ..)$ as $H^{i}_{?}(O_{k, S_{p}}, ..)$. We claim the following:
\begin{proposition}\label{proposition4.2}
There is a distinguished triangle
$$R\Gamma_{c}(O_{k,S_p}, M_p)\rightarrow R\Gamma_{f}(O_{k,S_p}, M_p)\rightarrow \oplus_{v\in S_p}R\Gamma_{f}(k_{v}, M_p)\oplus \oplus_{v\in S_\infty}R\Gamma(k_{v}, M_p)$$
\end{proposition}
\begin{proof}
As shown above we have a distinguished triangle for every finite place $v$ of $k$
$$R\Gamma_{f}(k_{v}, M_p)\rightarrow R\Gamma(k_{v}, M_p)\stackrel{g}{\rightarrow}R\Gamma_{/f}(k_{v}, M_p).$$
Since $R\Gamma_{/f}(k_{v}, M_p)=0$ when $v$ is infinite, we consider the following (non exact) sequence of complexes
$$R\Gamma(O_{k,S_p}, M_p)\stackrel{f}{\rightarrow} \oplus_{v\in S_{p}\cup S_{\infty}}R\Gamma(k_{v}, M_p)\stackrel{g}{\rightarrow}\oplus_{v\in S_{p}}R\Gamma_{/f}(k_{v}, M_p).$$
Lemma \ref{lemma4.1} gives a distinguished triangle which we tranlate by -1
 $$Cone(f)[-1]\rightarrow Cone(g\circ f)[-1] \rightarrow Cone(g)[-1]$$
which gives by definition the distinguished triangle
$$R\Gamma_{c}(O_{k,S_p}, M_p)\rightarrow R\Gamma_{f}(O_{k,S_p}, M_p)\rightarrow \oplus_{v\in S_p}R\Gamma_{f}(k_{v}, M_p)\oplus \oplus_{v\in S_\infty}R\Gamma(k_{v}, M_p)$$
$\Box$
\end{proof}
We denote By $E_p$ the field $E_p:=E\otimes \mathbb{Q}_p$. 
\begin{proposition}\label{proposition4.3}
Suppose that $p$ is odd or that $k$ is totally imaginary if $p=2$. Let $r<0$. There exists an $E_{p}[G]$-equivariant isomorphism
$$\vartheta^{r}_{F, S_p}: \Xi(E(r)_{F})\otimes E_p\; \xrightarrow{\;\sim\;} \; \mathrm{Det}_{E_p[G]}\mathrm{R\Gamma }_{c}(O_{k,S_p}, ((E(r)_{F}))_p) $$
\end{proposition}
\begin{proof}
The fundamental line is given by
$$\Xi(M):=\mathrm{Det}_{E[G]}(H^{0}_{\mathcal{M}}(M))\otimes \mathrm{Det}_{E[G]}(H^{1}_{\mathcal{M}}(M^{\vee}(1))^{\vee})\otimes \mathrm{Det}^{-1}_{E[G]}(M_{B}^{+})$$
We are interested in the motive $M=E(r)_{F}$. In that case $H^{0}_{\mathcal{M}}(M)=0$, $H^{1}_{\mathcal{M}}(M)=0$  and we have the following isomorphisms
\begin{itemize}
\item $M_{B}^{+}\otimes E_{p} \stackrel{\alpha}{\xrightarrow{\;\sim\;}} (\oplus_{v\in S_{\infty}}H^{0}(k_{v},M_{p}))$
\item The Chern map isomorphism $$H^{1}_{\mathcal{M}}(M^{\vee}(1))\otimes E_{p}\stackrel{ch}{\xrightarrow{\;\sim\;}}H^{1}_{f}(O_{k,S_{p}}, M_{p}^{\vee}(1)).$$
\end{itemize}
These two isomorphisms induce a third one
\begin{align*}
{(}\mathrm{Det}_{E[G]}(H^{1}_{\mathcal{M}}((M)^{\vee}(1))^{\vee})\otimes \mathrm{Det}^{-1}_{E[G]}(M_{B}^{+})\textbf{)}\otimes E_{p}\stackrel{\mathrm{Det}(ch^{\vee})\otimes \mathrm{Det^{-1}}(\alpha)}{\xrightarrow{\;\;\;\;\;\;\sim\;\;\;\;\;\;}}& \mathrm{Det}_{E_{p}[G]}(H^{1}_{f}(O_{k,S_{p}}, M_{p}^{\vee}(1))^{\vee})\\ &\otimes \mathrm{Det}^{-1}_{E_{p}[G]}(\oplus_{v\in S_{\infty}}R\Gamma(k_{v}, M_{p}))
\end{align*}
By Artin-Verdier duality we have
$$H^{i}_{f}(O_{k,S_{p}},M_{p})\simeq H^{3-i}_{f}(O_{k,S_{p}},M_{p}^{\vee}(1))^{\vee}$$
Thus, we can compute the cohomology of $R\Gamma_{f}(O_{k,S_{p}},M_p)$ in  terms of one degree (since degrees 0 and 1 are null in our case) and we have
$$\mathrm{Det}_{E_{p}[G]}R\Gamma_{f}(O_{k,S_{p}},M_p)=\mathrm{Det}_{E_{p}[G]}(H^{1}_{f}(O_{k,S_{p}}, M_{p}^{\vee}(1))^{\vee})$$
Let us write 
$$ \mathcal{J}_{S_p}: \mathrm{Det}_{E_{p}[G]}R\Gamma_{f}(O_{k,S_p}, M_p)\otimes \mathrm{Det}^{-1}_{E_{p}[G]} (\oplus_{v\in S_p}R\Gamma_{f}(k_{v}, M_p)\oplus \oplus_{v\in S_\infty}R\Gamma(k_{v}, M_p))\xrightarrow{\;\sim\;} \mathrm{Det}_{E_{p}[G]}R\Gamma_{c}(O_{k,S_p}, M_p)$$
for the isomorphism obtained from the distinguished triangle of Proposition \ref{proposition4.2}.
Since $r<0$, the complex $R\Gamma_{f}(k_{v}, M_p)$ is acyclic for all finite places (e.g. \cite{burnsFlach}, §2, after Lemma 1), and $\mathrm{Det}^{-1}_{E_{p}[G]} (\oplus_{v\in S_p}R\Gamma_{f}(k_{v}, M_p)=\otimes_{v\in S_p}\mathrm{Det}^{-1}_{E_{p}[G]} R\Gamma_{f}(k_{v}, M_p)$ maps to $E_{p}[G]$ in the following fashion (e.g. \cite{burnsFlach}, §2):\\
Write $V_{v}\rightarrow V_{v}$ for the complex $R\Gamma_{f}(k_{v}, M_p)$, then $\mathrm{Det}^{-1}_{E_{p}[G]} R\Gamma_{f}(k_{v}, M_p)$ maps to $E_{p}[G]$ via $$\mathrm{Id}_{V_{v}, triv}: \mathrm{det}^{-1}_{E_{p}[G]}V_{v}\otimes \mathrm{det}_{E_{p}[G]}V_{v} \xrightarrow{\;\sim\;} E_{p}[G].$$
and
$$\vartheta^{r}_{F, S_p}:= \mathcal{J}_{S_p}\circ(\mathrm{Det}(ch^{\vee})\otimes \mathrm{Det}^{-1}(\alpha)).$$
There is another way to map  $\Xi(E(r)_{F})\otimes E_p$ to $\mathrm{Det}_{E_p[G]}\mathrm{R\Gamma }_{c}(O_{k,S_p}, ((E(r)_{F}))_p)$ by mapping respective cohomology groups one by one (e.g. \cite{burnsFlach1}, §1.4, the remark after (1.16). This can be achieved as the following :\\
Since $R\Gamma_{f}(k_{v}, M_p)$ is acyclic and the complex $R\Gamma_{f}(O_{k,S_p}, M_p)$ is acyclic outside degree 2 ($r<0$), the exact sequence of Proposition  \ref{proposition4.2} gives
$$\left\{
\begin{array}{lr}

 H^{1}_{c}(O_{k,S_p}, M_p)\simeq \oplus_{v\in S_{\infty}}H^{0}(k_{v}, M_p)\\
 H^{2}_{c}(O_{k,S_p}, M_p)\simeq H^{2}_{f}(O_{k,S_p}, M_p)
\end{array}
\right.$$
which implies an isomorphism
$$ \tilde{\mathcal{J}}_{S_p}: \mathrm{Det}_{E_{p}[G]}R\Gamma_{f}(O_{k,S_p}, M_p)\otimes \mathrm{Det}^{-1}_{E_{p}[G]} (\oplus_{v\in S_\infty}R\Gamma(k_{v}, M_p))\xrightarrow{\;\sim\;} \mathrm{Det}_{E_{p}[G]}R\Gamma_{c}(O_{k,S_p}, M_p)$$
and this provides us with a new isomorphic map
$$\tilde{\vartheta}^{r}_{F, S_p}:\Xi(E(r)_{F})\otimes E_p\; \xrightarrow{\;\sim\;} \; \mathrm{Det}_{E_p[G]}\mathrm{R\Gamma }_{c}(O_{k,S_p}, M_p) $$
where $\tilde{\vartheta}^{r}_{F, S_p}:= \tilde{\mathcal{J}}_{S_p}\circ(\mathrm{Det}(ch^{\vee})\otimes \mathrm{Det}^{-1}(\alpha)),$
in which we implicitly use the map
$$\mathrm{det}_{E_{p}[G]}^{-1}H^{0}_{f}(k_{v}, M_p)\otimes \mathrm{det}_{E_{p}[G]}H^{1}_{f}(k_{v}, M_p) \widetilde{\rightarrow} E_{p}[G]$$
to get rid of $f$-cohomology.\\
Note that the maps $\vartheta^{r}_{F, S_p}$ and $\tilde{\vartheta}^{r}_{F, S_p}$ only differ by the way $\otimes_{v\in S_p}\mathrm{Det}^{-1}_{E_{p}[G]} R\Gamma_{f}(k_{v}, M_p)$ maps to $E_{p}[G]$. In fact one has (\cite{burnsFlach} , (11), (12))
$$\tilde{\vartheta}^{r}_{F, S_p}=\varepsilon_{S_p}(r).\vartheta^{r}_{F, S_p}$$
and $$\varepsilon_{S_p}(r)=\prod_{v\in S_p}(1-Nv^{-r}\mathrm{Frob}_{v}e_{I_{v}})\in (E[G])^{\times}.$$
where $\mathrm{Frob}_{v}$ is any representative in $G$ of the Frobenius map  associated to the finite place $v$, and $e_{I_{v}}$ is the idempotent of $\mathbb{Q}[G]$ which corresponds to the inertia subgroup $I_v$.\\
Note that $\varepsilon_{S_p}(r)$ is the inverse of the factor used in \cite{burnsFlach}, since we adopted a dual formulation of the ETNC compared to  Loc. cit.
$\Box$
\end{proof}
Let $\nu$ be any integer. We define the lattice
$$\mathcal{I}^{p}_{F, S_p}(r):=\mathrm{Det}_{O_p[G]}\mathrm{R\Gamma }_{c}(O_{k,S_p}, p^{\nu}T_p) \subset \mathrm{Det}_{E_p[G]}\mathrm{R\Gamma }_{c}(O_{k,S_p}, M_p).$$
The definition of $\mathcal{I}^{p}_{F, S_p}(r)$ doesn't depend on the choice of $\nu$ (e.g. \cite{burnsFlach}, Prop. 1.20) and we set
$$\Xi(E(r)_{F})_{O}:=\cap_p (\Xi(E(r)_{F})\cap (\vartheta^{r}_{F, S_p})^{-1}(\mathcal{I}^{p}_{F, S_p}(r))).$$
\subsubsection{The isomorphism  $\vartheta^{r}_{F, \infty}$}\label{4.2.2}
\begin{definition}
For any isomorphism of finitely generated $R$-modules $\phi : V\xrightarrow{\;\sim\;} W$, we let
$$\phi_{triv}: \mathrm{Det}^{-1}_{R}(V)\otimes \mathrm{Det}_{R}(W)\xrightarrow{\;\sim\;} R,$$
obtained by the following composition of isomorphisms
$$\mathrm{Det}^{-1}_{R}(V)\otimes \mathrm{Det}_{R}(W)\stackrel{\mathrm{Det}^{-1}_{R}(\phi)\otimes \mathrm{id}}{\xrightarrow{\;\;\;\;\sim\;\;\;\;}}\mathrm{Det}^{-1}_{R}(W)\otimes \mathrm{Det}_{R}(W)\tilde{\rightarrow} R.$$
\end{definition}
The Beilinson regulator map induces an isomorphism
$$\rho^{r}_{F}: K_{1-2r}(F)\otimes \mathbb{R}\;\xrightarrow{\;\sim\;}\; Y_{r}(F)^{+}\otimes \mathbb{R},$$
which defines the isomorphism
$$\vartheta^{r}_{F, \infty} := (({\rho ^{r}_{F}})^{\#})_{triv} : \Xi(\mathbb{Q}(r)_{F})\otimes \mathbb{R}\; \xrightarrow{\;\sim\;}\; \mathbb{R}[G]^{\#}.$$
For the motive $E(r)_F$, extension of scalars gives
$$\vartheta^{r}_{F, \infty}: \Xi(E(r)_{F})\otimes \mathbb{R}\; \xrightarrow{\;\sim\;}\; \mathbb{R}\otimes E[G]^{\#}.$$
\subsection{Statement of the ETNC}\label{4.3}
Let $S$ be any finite set of places of $k$ containing the set $S_\infty$. We write $L_{S}(E(r)_{F},s)$ for the $S$-truncated $\mathbb{C} [G]$-valued $L$-function of the motive $E(r)_{F}$.
\noindent Then with respect to the canonical identification $\mathbb{C} [G]=\prod_{\hat{G}}\mathbb{C}$
we have
$$L_{S}(E(r)_{F},s) = (L_{S}(s+r, \chi))_{\chi\in\hat{G}}.$$
\subsubsection*{Remark}
$$L^{*}(E(r)_{F},0):= L_{S_{\infty}}^{*}(E(r)_{F},0)\in \mathbb{R}[G]^{\times}$$\\
The statement of the Equivariant Tamagawa Number conjecture for the motive $E(r)_{F}$ is given as follows

\begin{conjecture}[Conjecture \ref{conjecture4.3} (ETNC)]\label{conjecture4.3}

One has 
\begin{conjecture}[Conjecture \ref{conjecture4.3.1}.1]\label{conjecture4.3.1}{\textit{(Rationality Conjecture)}}
$$(\vartheta^{r}_{F, \infty})^{-1}(L^{*}(E(r)_{F},0)^{-1}).E[G] \supseteq \Xi(E(r)_{F})$$
\end{conjecture}
Further, each of the following equivalent conjectures is true
\begin{conjecture}[Conjecture \ref{conjecture4.3.2}.2]\label{conjecture4.3.2}
$$O_{p}[G]\vartheta^{r}_{F, S_{p}} \circ (\vartheta^{r}_{F, \infty})^{-1} (L^{*}(E(r)_{F},0)^{-1})= \mathcal{I}^{p}_{F, S}(r)$$
\end{conjecture}
\begin{conjecture}[Conjecture \ref{conjecture4.3.3}.3]\label{conjecture4.3.3}
$$ O[G](\vartheta^{r}_{F, \infty})^{-1} (L^{*}(E(r)_{F},0)^{-1})=\Xi(E(r)_{F})_{O}$$
\end{conjecture}
\begin{conjecture}[Conjecture \ref{conjecture4.3.4}.4]\label{conjecture4.3.4}
$$O_{p}[G]\vartheta^{r}_{F, S_{p}} \circ (\vartheta^{r}_{F, \infty})^{-1} (L^{*}(E(r)_{F},0)^{-1})=\mathrm{Det}_{O_{p}[G]}\mathrm{R}\Gamma_{c}(O_{k,S_p}, T_p)$$
\end{conjecture}
\end{conjecture}
\begin{remarks}
\item[1] In Rationality Conjecture (Conjecture \ref{conjecture4.3.1}.1), the inclusion is inverted compared to the one in the Rationality Conjecture stated in (\cite{burnsGreither}, Conjecture 3.1). In fact, as we mentioned above, we used the parity convention  of \cite{KM}. Our formulation is then dual to that in \cite{burnsGreither} but are both equivalent nonetheless.
\item[2] The Conjecture \ref{conjecture4.3} is compatible with enlarging $S_p$ to any finite set of primes containing the infinite primes (e.g. \cite{LeungKings}, Lemma 2.3).
\item[3] Conjecture \ref{conjecture4.3.1}.1 is equivalent to Stark's conjecture for $r=0$ and $E=\mathbb{Q}$, it is also equivalent to the central conjecture formulated by Gross in \cite{Gross1}.
\item[4] Each  of the equalities in Conjectures \ref{conjecture4.3.2}.2, \ref{conjecture4.3.3}.3 and \ref{conjecture4.3.4}.4 is equivalent to the "Lifted Root Number Conjecture" formulated by Gruenberg \& al \cite{GrunenbergAl} but only when $E=\mathbb{Q}$.
\end{remarks}
 \subsection{Burns-De jeu-Gangl's conjecture and ETNC in the case of an imaginary quadratic base field}
Let us work first with the non integral case (Conjecture \ref{conjecture4.3.1}.1). We have the following
\begin{theorem}\label{theorem4.4}
Let $F/k$ be an abelian extension over the imaginary quadratic number field $k$. Suppose that $E$ is a number field containing all values of characters $\chi$ of $G:=\mathrm{Gal}(F/k)$ and that $r$ is a stricty negative integer. We make the identification $e_{\chi}E\cong E$.\\
Then the following statements are equivalent
\begin{itemize}
\item[1]Conjecture \ref{conjecture4.3.1}.1 holds for the pair $(E(r)_{F}, E[G])$.
\item[2]For each character $\chi$ of $G$ there exists an element $\tilde{\epsilon}_{\chi}(F)\in e_{\chi}(K_{1-2r}(F)\otimes E)$ which verifies
$$\rho_{F}^{r}(\tilde{\epsilon}_{\chi}(F)) = L^{'}(r,\chi^{-1})$$
where the Beilinson Regulator $\rho_{F}^{r}$ is defined here over $K_{1-2r}(F)\otimes E$ by extension of scalars.
\end{itemize}
\end{theorem}
\begin{proof}
Conjecture \ref{conjecture4.3.1}.1  reads 
$$(\vartheta^{r}_{F, \infty})^{-1}(L^{*}(E(r)_{F},0)^{-1}).E[G] \supseteq \Xi(E(r)_{F})$$
where $$\Xi(E(r)_{F}=\mathrm{Det}^{-1}_{E[G]}(K_{1-2r}(F)\otimes E)^{\#}\otimes \mathrm{Det}_{E[G]}(Y_{r}(F)^{+}\otimes E)^{\#}.$$

By  § \ref{4.2.2} this is equivalent to
$$ L^{*}(E(r)_{F},0)^{-1}.E[G]\supseteq ((\rho^{r}_{F})^{\#})_{triv}(\mathrm{Det}^{-1}_{E[G]}(K_{1-2r}(F)\otimes E)^{\#}\otimes\mathrm{Det}_{E[G]}(Y_{r}(F)^{+}\otimes E)^{\#})$$
Suppose now that $k$ is totally imaginary. Then $L^{*}(E(r)_{F},0)=L^{'}(E(r)_{F},0)$ and we canonically identify
$$Y_{r}(F)^{+}\otimes E\cong E[G]^{\#}.$$
This identification sends  for each $\sigma\in G$, seen as an embedding $\sigma:F\rightarrow \mathbb{C}$, the element $(0,..,(2\pi i)^{-r}a_{\sigma},..,(2\pi i)^{-r}a_{\tau\sigma},..,0)$ (with $a_{\sigma}=\pm a_{\tau\sigma})$ to $a_{\sigma}\sigma^{-1}$.
Hence
$$\mathrm{Det}_{E[G]}(Y_{r}(F)^{+}\otimes E)^{\#}=E[G]$$
Therefore, Conjecture \ref{conjecture4.3.1}.1 is equivalent to
$$ (L^{*}(E(r)_{F},0)^{-1})^{\#}E[G]\supseteq (\rho^{r}_{F})_{triv}(\mathrm{Det}^{-1}_{E[G]}(K_{1-2r}(F)\otimes E)\otimes E[G])$$
Let $\chi\in\mathrm{Hom}(G,\mathbb{C})$ be a character of $G$ and $e_{\chi}$ the corresponding idempotent. Since $E$ contains all values of such characters of $G$, we get $e_{\chi}E[G] = e_{\chi}E$. In the following we use the identification $e_{\chi}E\cong E$. \\
By § \ref{4.3} one can also identify  $L^{'}(E(r)_{F},0)=\Sigma_{\chi\in\mathrm{Hom}(G,\mathbb{C}}L^{'}(r,\chi)e_{\chi}$. Conjecture \ref{conjecture4.3.1}.1 is then equivalent to the following being true for all characters of $G$
$$ L^{'}(r, \chi^{-1})^{-1}E \supseteq (\rho^{r}_{F})_{triv}(\mathrm{Det}^{-1}_{E}(e_{\chi}(K_{1-2r}(F)\otimes E))\otimes E)$$
The number fields $k$ is totally imaginary, thus by Beilinson regulator
$$(K_{1-2r}(F)\otimes E)\otimes \mathbb{R}\simeq E[G]\otimes \mathbb{R}$$
The $E$-vector space $e_{\chi}(K_{1-2r}(F)\otimes E)$ is then necessarily of dimension 1 and  the restriction of $(\rho^{r}_{F})_{triv}$ to the $\chi$-eigenspaces is defined as follows
\begin{align*}
(\rho^{r}_{F})_{triv}&: \mathrm{Det}^{-1}_{E}(e_{\chi}(K_{1-2r}(F)\otimes E))\otimes E \rightarrow E\\
&:\mathrm{Hom}(e_{\chi}(K_{1-2r}(F)\otimes E), E)\otimes E\stackrel{(({\rho^{r}_{F}})^{\vee})^{-1}\otimes Id_{E}}{\rightarrow} E
\end{align*}
It follows that Conjecture \ref{conjecture4.3.1}.1 is then equivalent to the following being true for all characters of $G$
$$ L^{'}(r, \chi^{-1})E \subseteq \rho^{r}_{F}(e_{\chi}(K_{1-2r}(F)\otimes E))$$
which means there exists for each character $\chi$ of $G$ an element $\tilde{\epsilon}_{\chi}(F)\in e_{\chi}(K_{1-2r}(F)\otimes E)$ which verifies
$$\rho^{r}_{F}(\tilde{\epsilon}_{\chi}(F)) = L^{'}(r,\chi^{-1}).$$

$\Box$
\end{proof}
\begin{corollary}\label{corollary4.5}
Let $F/k$ be a finite abelian extension with $k$ imaginary quadratic. Suppose that $r<0$ and that $S$ is a finite set of places of $k$ containing the set of infinites places. Let $E$ be a number field which contains all values of characters $\chi$ of $G$, and $O$ be its ring of integers.\\
If Conjecture \ref{conjecture4.3.1}.1 holds for the pair $(E(r)_{F}, E[G])$, then, the first statement of Conjecure \ref{conjecture3.1} holds for the set of data $(F/k, r, S, O, p)$, for all rational primes $p$.
\end{corollary}
\begin{proof}
Recall that $\mathcal{T}$ denotes the complete $\mathbb{Z}$-lattice in $\mathbb{R}[G]$, which is axactly the image of $K_{1-2r}(F)$ by the Beilinson regulator.\\
If Conjecture \ref{conjecture4.3.1}.1  holds for the pair $(E(r)_{F}, E[G])$, then, Theorem \ref{theorem4.4}, shows that the element $w_{1-r}(F^{\mathrm{ker}(\chi)})L^{'}(r,\chi^{-1})\chi^{-1}(\sigma)\mid{G}\mid e_{\chi}$ belongs to $\mathcal{T}\otimes E$.\\
Consequently, the first statement of Conjecure \ref{conjecture3.1} holds for the set of data $(F/k, r, S_\infty, O, p)$, for all rational primes $p$. But, this also means that the first statement of Conjecure 3.1 holds for the set of data $(F/k, r, S, O, p)$, for all rational primes $p$ and any finite set of places of $k$ containing the infinite places since
$$L^{'}_{S}(r,\chi)=\prod_{v\in S\backslash  S_\infty}(1-Nv^{-r}\chi^{-1}(v))L^{'}_{S_\infty}(r,\chi)=\prod_{v\in S\backslash S_\infty}(1-Nv^{-r}\chi^{-1}(v))L^{'}(-r,\chi)$$
and $\prod_{v\in S\backslash S_\infty}(1-Nv^{-r}\chi^{-1}(v))\in O$.\hspace*{\fill}$\Box$
\end{proof}
Let us return now to the integral case using conjecture \ref{conjecture4.3.4}.4 for more results. We need to introduce some results first.\\
In the rest of the paper we suppose that the integer $r$ is always strictly negative and we fix $T_p =O_{p}[G](r)$.
\begin{lemma}\label{lemma4.6}
Let $p$ be a rational prime and $S$ a finite set of places of $k$ containing the infinite places, the $p$-places and the places which ramify in $F/k$.\\
Suppose that either $p$ is odd or $k$ is totally imaginary.
There is a distinguished triangle 
\begin{equation}\label{1}
R\Gamma_{c}(O_{k,S}, T_{p})\rightarrow R\Gamma(O_{k,S}, T_{p}^{\vee}(1))^{\vee}[-3]\rightarrow (\prod_{\sigma:k\rightarrow \mathbb{C}}T_{p})^{+}[0]
\end{equation}
\end{lemma}
\begin{proof}
The distinguished triangle \ref{1} is the same as the one given in (\cite{burnsGreither}, §3.4). It arises in the following way: The exact sequence (Poitou-Tate global duality, e.g.\cite{milne})
$$
 \xymatrix{
&0 \ar[r] &H^{0}(O_{k,S}, T_p) \ar[r] &\oplus_{v\in S_{f}}H^{0}(k_{v}, T_{p})\oplus \oplus_{v\in S_{\infty}}\hat{H}^{0}(k_{v}, T_{p}) \ar[r] &H^{2}(O_{k,S}, T_{p}^{*}(1))^{*} \ar[d] & \\
&  &H^{1}(O_{k,S}, T_{p}^{*}(1))^{*} \ar[d]  &\oplus_{v\in S}H^{1}(k_{v}, T_{p})\ar[l] &H^{1}(O_{k,S}, T_{p}) \ar[l] & \\
&  &H^{2}(O_{k,S}, T_{p}) \ar[r]  &\oplus_{v\in S}H^{2}(k_{v}, T_{p})\ar[r] &H^{0}(O_{k,S}, T_{p}^{*}(1))^{*} \ar[r] &0
}
$$
where $(){*}$ denotes the Pontryagin dual, and $S_f$ the subset of finite places inside $S$.\\
 The sequence above gives a distinguished triangle whenever $p$ is odd or $k$ is totally imaginary
$$R\Gamma(O_{k,S}, T_p)\rightarrow \oplus_{v\in S_{f}}R\Gamma(k_{v}, T_{p})\rightarrow R\Gamma(O_{k,S}, T_p^{*}(1))^{*}[-2]$$
By (\cite{burnsFlach}, Lemma 16 (b)), one can see that
$$R\Gamma(O_{k,S}, T_{p}^{*}(1))^{*}\cong R\Gamma(O_{k,S}, T_{p}^{\vee}(1))^{\vee}$$
this along with the application of Lemma 4.1 to the (non distinguished) triangle
$$R\Gamma(O_{k,S}, T_p)\stackrel{f}{\rightarrow }\oplus_{v\in S}R\Gamma(k_{v}, T_{p})
\stackrel{g}{\rightarrow }\oplus_{v\in S_{f}}R\Gamma(k_{v}, T_{p})$$
gives the triangle \ref{1}.\\
$\Box$
\end{proof}
By Lemma \ref{lemma4.6} we get an isomorphism
\begin{equation}\label{2}
\mathrm{Det}_{O_{p}[G]}R\Gamma(O_{k,S}, T_{p}^{\vee}(1))^{\vee}[-3]\otimes \mathrm{Det}^{-1}_{O_{p}[G]}(\prod_{\sigma:k\rightarrow \mathbb{C}}T_{p})^{+}[0]\stackrel{\Im}{\rightarrow}\mathrm{Det}_{O_{p}[G]}R\Gamma_{c}(O_{k,S}, T_p)
\end{equation}
The isomorphism $\Im \otimes \mathrm{Id}_{E_p}$  explicitly maps the following cohomology groups
$$\left\{
\begin{array}{lr}

 H^{1}_{c}(O_{k,S}, M_p)\simeq \oplus_{v\in S_{\infty}}H^{0}(k_{v}, M_p)\\
 H^{2}_{c}(O_{k,S}, M_p)\simeq H^{1}(O_{k,S}, M_p^{\vee}(1))^{\vee}
\end{array}
\right.$$
Let us go even further: If we apply Lemma \ref{lemma4.1} to the (non exact) triangle
$$\oplus_{v\in S_{f}}R\Gamma_{f}(k_{v}, M_{p})\oplus\oplus _{v\in S_{\infty}}R\Gamma(k_{v}, M_{p})\stackrel{f}{\rightarrow }\oplus_{v\in S_{\infty}}R\Gamma(k_{v}, M_{p})\stackrel{g}{\rightarrow }R\Gamma_{c}(O_{k,S}, M_p)[1]$$
we get a distinguished triangle
$$\oplus_{v\in S_f}R\Gamma_{f}(k_{v}, M_{p})\rightarrow R\Gamma_{f}(O_{k,S}, M_p)\rightarrow R\Gamma(O_{k,S}, M_{p}^{\vee}(1))^{\vee}[-3]$$
This latter shows that the cohomology of $R\Gamma(O_{k,S}, M_{p}^{\vee}(1))^{\vee}[-3]$ 
identifies with $f$-cohomology given that $R\Gamma_{f}(k_{v}, M_{p})$ is acyclic
\begin{equation}\label{5}
H^{1}(O_{k,S}, M_p^{\vee}(1))^{\vee}\simeq H^{2}_{f}(O_{k,S}, M_p)
\end{equation}
A second isomorphism then ensues
\small{
\begin{equation}\label{3}
\begin{small}
\end{small}
\mathrm{Det}_{E_{p}[G]}R\Gamma(O_{k,S}, M_{p}^{\vee}(1))^{\vee}[-3]\otimes \mathrm{Det}^{-1}_{E_{p}[G]}(\prod_{\sigma:k\rightarrow \mathbb{C}}M_{p})^{+}[0]\stackrel{\mho}{\rightarrow}\mathrm{Det}_{E_{p}[G]}R\Gamma_{f}(O_{k,S}, M_{p})\otimes\otimes_{v\in S_{\infty}} (\mathrm{Det}_{E_{p}[G]}^{-1}R\Gamma(k_{v}, M_{p}))
\end{equation}
}
\normalsize{To} take the enlarged set of primes $S$ into account in the formulation of the ETNC we recall the distinguished triangle (e.g. \cite{LeungKings}, Lemma 2.3)
\begin{equation}
R\Gamma_{c}(O_{k, S}, M_{p})\rightarrow R\Gamma_{c}(O_{k, S_{p}}, M_{p}) \rightarrow
\oplus_{v\in S_{f}\backslash S_{p}} R\Gamma_{f}(k_{v}, M_{p})
\label{6}
\end{equation}


One gets then an isomorphism
\begin{equation}
\mathrm{Det}_{E_{p}[G]}R\Gamma_{c}(O_{k, S}, M_{p})\stackrel{\tilde{\mathcal{r}}}{\rightarrow}\mathrm{Det}_{E_{p}[G]}R\Gamma_{c}(O_{k, S_{p}}, M_{p})
\label{4}
\end{equation}
We use the results ((\ref{1}),..,(\ref{4})) above to define the diagram in the proposition below:
\begin{proposition}\label{proposition4.7}
The following diagram is commutative (all arrows are isomorphisms)
$$
\xymatrixcolsep{2pc}\xymatrix{
&\mathrm{Det}_{E_{p}[G]}R\Gamma(O_{k,S}, M_{p}^{\vee}(1))^{\vee}[-3]\otimes \mathrm{Det}^{-1}_{E_{p}[G]}(\prod_{\sigma:k\rightarrow \mathbb{C}}M_{p})^{+}[0]\ar[d]^{\mho}\ar[r]^<<<<{\Im\otimes \mathrm{Id}_{E_p}} &\mathrm{Det}_{E_{p}[G]}R\Gamma_{c}(O_{k,S}, M_p)\ar[d]^{\tilde{\mathcal{r}}}\\
&\mathrm{Det}_{E_{p}[G]}R\Gamma_{f}(O_{k,S}, M_{p})\otimes\otimes_{v\in S_{\infty}} (\mathrm{Det}^{-1}_{E_{p}[G]}R\Gamma(k_{v}, M_{p}))\ar[d]^{\jmath} &\mathrm{Det}_{E_{p}[G]}R\Gamma_{c}(O_{k,S_p}, M_p)\ar[d]^{({\tilde{\vartheta}^{r}}_{F, S_p})^{-1}} \\
&\mathrm{Det}_{E_{p}[G]}R\Gamma_{f}(O_{k,S_{p}}, M_{p})\otimes\otimes_{v\in S_{\infty}} (\mathrm{Det}^{-1}_{E_{p}[G]}R\Gamma(k_{v}, M_{p})) \ar[r]^<<<<<<<{\mathcal{H}} & \Xi(E(r)_{F})\otimes E_p
}
$$
where $\jmath$ identifies cohomology of  $R\Gamma_{f}(O_{k,S}, M_{p})$ with the cohomology of $R\Gamma_{f}(O_{k,S_p}, M_{p})$ and $$\mathcal{H}=({\tilde{\vartheta}^{r}}_{F, S_p})^{-1}\circ \tilde{\mathcal{J}}_{S_p}$$
where the isomorphisms ${\tilde{\vartheta}^{r}}_{F, S_p}$ and $\tilde{\mathcal{J}}_{S_p}$ are explained in the proof of Proposition \ref{proposition4.3}.

\end{proposition}
\begin{proof}
The commutativity of the diagram above can be shown easily when passing to cohomology. Note that since $E_{p}[G]$ is semi-simple, all the involved (finitely generated) cohomology groups are perfect.\\
By definition of the isomorphism $\mathcal{H}$ one has a commutative diagram
$$
\xymatrixcolsep{2pc}\xymatrix{
& &\mathrm{Det}_{E_{p}[G]}R\Gamma_{c}(O_{k,S_p}, M_p)\ar[d]^{({\tilde{\vartheta}^{r}}_{F, S_p})^{-1}} \\
&\mathrm{Det}_{E_{p}[G]}R\Gamma_{f}(O_{k,S_{p}}, M_{p})\otimes\otimes_{v\in S_{\infty}} (\mathrm{Det}^{-1}_{E_{p}[G]}R\Gamma(k_{v}, M_{p}))\ar[ru]^{\tilde{\mathcal{J}}_{S_p}} \ar[r]^<<<<<<<{\mathcal{H}} & \Xi(E(r)_{F})\otimes E_p
}
$$
Next we shall define the isomorphism $\jmath$.
By the proof of proposition \ref{proposition4.3}, the isomorphism $\tilde{\mathcal{J}}_{S_p}$ arises from the distinguished triangle 
$$R\Gamma_{c}(O_{k,S_p}, M_p)\rightarrow R\Gamma_{f}(O_{k,S_p}, M_p)\rightarrow \oplus_{v\in S_p}R\Gamma_{f}(k_{v}, M_p)\oplus \oplus_{v\in S_\infty}R\Gamma(k_{v}, M_p)$$
Passing to cohomology, $\tilde{\mathcal{J}}_{S_p}$ links the following cohomology groups
$$\left\{
\begin{array}{lr}

 H^{1}_{c}(O_{k,S_p}, M_p)\simeq \oplus_{v\in S_{\infty}}H^{0}(k_{v}, M_p)\\
 H^{2}_{c}(O_{k,S_p}, M_p)\simeq H^{2}_{f}(O_{k,S_p}, M_p)
\end{array}
\right.$$
We can use the same arguments of proposition \ref{proposition4.3} to get a distinguished triangle
$$R\Gamma_{c}(O_{k,S}, M_p)\rightarrow R\Gamma_{f}(O_{k,S}, M_p)\rightarrow \oplus_{v\in S_{f}}R\Gamma_{f}(k_{v}, M_p)\oplus \oplus_{v\in S_\infty}R\Gamma(k_{v}, M_p)$$
where $S_{f}$ is the set of finite places of $S$.
This gives us an isomorphism
$$ \tilde{\mathcal{J}}_{S}: \mathrm{Det}_{E_{p}[G]}R\Gamma_{f}(O_{k,S}, M_p)\otimes \mathrm{Det}^{-1}_{E_{p}[G]} (\oplus_{v\in S_\infty}R\Gamma(k_{v}, M_p))\xrightarrow{\;\sim\;} \mathrm{Det}_{E_{p}[G]}R\Gamma_{c}(O_{k,S_p}, M_p)$$
which explicitly maps the following cohomology groups
$$\left\{
\begin{array}{lr}

 H^{1}_{c}(O_{k,S}, M_p)\simeq \oplus_{v\in S_{\infty}}H^{0}(k_{v}, M_p)\\
 H^{2}_{c}(O_{k,S}, M_p)\simeq H^{2}_{f}(O_{k,S}, M_p)
\end{array}
\right.$$
By the distinguished triangle (\ref{6}), we have an isomorphism (this isomorphism induces the isomorphism $\tilde{\mathcal{r}}$)
$$H^{2}_{c}(O_{k,S}, M_p)\simeq H^{2}_{c}(O_{k,S_{p}}, M_p)$$
hence $$H^{2}_{f}(O_{k,S}, M_p)\simeq  H^{2}_{f}(O_{k,S_{p}}, M_p)$$
The isomorphism $\;\jmath\;$ is induced by this isomorphism and the identity on $\otimes_{v\in S_{\infty}}\mathrm{Det}^{-1}_{E_{p}[G]}R\Gamma(k_{v}, M_p)$. Therefore, we have (by definition of $\jmath $) the following commutative diagram
$$
\xymatrixcolsep{2pc}\xymatrix{
& &\mathrm{Det}_{E_{p}[G]}R\Gamma_{c}(O_{k,S}, M_p)\ar[d]^{\tilde{\mathcal{r}}}\\
&\mathrm{Det}_{E_{p}[G]}R\Gamma_{f}(O_{k,S}, M_{p})\otimes\otimes_{v\in S_{\infty}} (\mathrm{Det}^{-1}_{E_{p}[G]}R\Gamma(k_{v}, M_{p}))\ar[ru]^{\tilde{\mathcal{J}}_{S}}\ar[d]^{\jmath} &\mathrm{Det}_{E_{p}[G]}R\Gamma_{c}(O_{k,S_p}, M_p) \\
&\mathrm{Det}_{E_{p}[G]}R\Gamma_{f}(O_{k,S_{p}}, M_{p})\otimes\otimes_{v\in S_{\infty}} (\mathrm{Det}^{-1}_{E_{p}[G]}R\Gamma(k_{v}, M_{p}))\ar[ru]^{\tilde{\mathcal{J}}_{S_{p}}}& 
}
$$
It remains to show that the following diagram commutes
$$
\xymatrixcolsep{2pc}\xymatrix{
&\mathrm{Det}_{E_{p}[G]}R\Gamma(O_{k,S}, M_{p}^{\vee}(1))^{\vee}[-3]\otimes \mathrm{Det}^{-1}_{E_{p}[G]}(\prod_{\sigma:k\rightarrow \mathbb{C}}M_{p})^{+}[0]\ar[d]^{\mho}\ar[r]^<<<<{\Im\otimes \mathrm{Id}_{E_p}} &\mathrm{Det}_{E_{p}[G]}R\Gamma_{c}(O_{k,S}, M_p)\\
&\mathrm{Det}_{E_{p}[G]}R\Gamma_{f}(O_{k,S}, M_{p})\otimes\otimes_{v\in S_{\infty}} (\mathrm{Det}^{-1}_{E_{p}[G]}R\Gamma(k_{v}, M_{p}))\ar[ru]^{\tilde{\mathcal{J}}_{S}} &
}
$$
for, this we recall again that $\tilde{\mathcal{J}}_{S}$ maps the cohomology groups
$$\left\{
\begin{array}{lr}

 H^{1}_{c}(O_{k,S}, M_p)\simeq \oplus_{v\in S_{\infty}}H^{0}(k_{v}, M_p)\\
 H^{2}_{c}(O_{k,S}, M_p)\simeq H^{2}_{f}(O_{k,S}, M_p)
\end{array}
\right.
$$
while, as stated above, $\Im \otimes \mathrm{Id}_{E_p}$   maps the following cohomology groups
$$\left\{
\begin{array}{lr}

 H^{1}_{c}(O_{k,S}, M_p)\simeq \oplus_{v\in S_{\infty}}H^{0}(k_{v}, M_p)\\
 H^{2}_{c}(O_{k,S}, M_p)\simeq H^{1}(O_{k,S}, M_p^{\vee}(1))^{\vee}
\end{array}
\right.$$
This induces an isomorphism
$$H^{1}(O_{k,S}, M_p^{\vee}(1))^{\vee}\simeq H^{2}_{f}(O_{k,S}, M_p)$$
which is exactly the isomorphism (\ref{5}) used to induce the isomorphism $\mho$ in (\ref{3}). Thus
$$\mho= \tilde{\mathcal{J}}_{S}\circ (\Im \otimes \mathrm{Id}_{E_p})^{-1}$$
$\Box$
\end{proof}
For later computations, we need a detailed description of how the isomorphism $\mathcal{H}^{-1}$ works in more general terms.\\
In the proof of proposition \ref{proposition4.3} we mapped motivic cohomology to $f$-cohomolgy 
provided that the complex $R\Gamma_{f}(O_{k,S_{p}},M_{p})$ has non-trivial cohomolgy only in degree 2 (that is the group $H^{2}_{f}(O_{k,S_{p}},M_{p})\cong H^{1}_{f}(O_{k,S_{p}},M_{p}^{\vee}(1))^{\vee}$. The other two degrees isomorphisms are
\begin{itemize}
\item The cycle class map isomorphism
   $$H^{0}_{\mathcal{M}}(M)\otimes E_{p}\stackrel{cyc}{\xrightarrow{\;\sim\;}} H^{0}_{f}(O_{k,S_{p}}, M_{p})$$
\item The chern class map isomorphism
$$H^{1}_{\mathcal{M}}(M)\otimes E_{p}\stackrel{ch}{\xrightarrow{\;\sim\;}} H^{1}_{f}(O_{k,S_{p}}, M_{p})$$
(remember that $H^{0}_{\mathcal{M}}(M)=0$ and $H^{1}_{\mathcal{M}}(M)=0$  for our motive, but we need to compute the above isomorphisms nonetheless to take account of the finite groups which appear in the integral case).
\end{itemize}
Hence, more accurately
$$\mathcal{H}^{-1}:=\mathrm{Det}(ch^{\vee})\otimes \mathrm{Det}^{-1}(ch) \otimes \mathrm{Det}(cyc) \otimes \mathrm{Det}^{-1}(cyc^{\vee})\otimes\mathrm{Det}^{-1}(\alpha)$$
these maps are summarized as 
$$
\xymatrixcolsep{0.0pc}\xymatrix{
&\mathrm{Det}^{-1}_{E_{p}[G]}(H^{0}_{\mathcal{M}}(M^{\vee}(1))^{\vee}\otimes E_{p})\ar[d]^{\mathrm{Det}^{-1}(cyc^{\vee})}&\otimes \mathrm{Det}_{E_{p}[G]}(H^{0}_{\mathcal{M}}(M)\otimes E_{p})\ar[d]^{\mathrm{Det}(cyc)}&\otimes \mathrm{Det}^{-1}_{E_{p}[G]}(H^{1}_{\mathcal{M}}(M)\otimes E_{p})\ar[d]^{\mathrm{Det}^{-1}(ch)}\\
 &\mathrm{Det}^{-1}_{E_{p}[G]}H^{0}_{f}(O_{k,S_{p}}, M_{p}^{\vee}(1))^{\vee}) &\otimes\mathrm{Det}_{E_{p}[G]}H^{0}_{f}(O_{k,S_{p}}, M_{p})&\otimes \mathrm{Det}^{-1}_{E_{p}[G]}H^{1}_{f}(O_{k,S_{p}}, M_{p}) 
}
$$
and
$$
\xymatrixcolsep{0.0pc}\xymatrix{
&\mathrm{Det}_{E_{p}[G]}(H^{1}_{\mathcal{M}}((M)^{\vee}(1))^{\vee})\ar[d]^{\mathrm{Det}(ch^{\vee})} &\otimes \mathrm{Det}^{-1}_{E_{p}[G]}(M_{B}^{+}\otimes E_{p})\ar[d]^{\mathrm{Det}^{-1}(\alpha)}\\
& \mathrm{Det}_{E_{p}[G]}H^{2}_{f}(O_{k,S_{p}}, M_{p})&\mathrm{Det}^{-1}_{E_{p}[G]}(\oplus_{\sigma: k\rightarrow \mathbb{C}}M_{p})^{+}
}
$$

\begin{theorem}\label{theorem4.8}
Let $F/k$ be a finite abelian extension of number fields with $k$ imaginary quadratic. Suppose that $p$ is a rational prime which does not divide $\# G$, that $E$ is a number field which contains all values of characters of $G$ and let $r$ denote a strictly negative integer. Let $S$ be a finite set of places of $k$ containing the infinite places, the $p$-places and the places which ramify in $F/k$ and $S_f$ the subset of finite places of $S$. Then the following statements are equivalent
\begin{itemize}
\item Conjecture \ref{conjecture4.3.4}.4 holds for the pair $(E(r)_{F}, O_{p}[G])$.
\item For all $p$-adic characters $\chi$ of $G$ one has
\small{
\begin{align*}
&\rho_{F}^{r}(e_{\chi}(K_{1-2r}(F)_{/tors}\otimes O_{p}))=\\ &\prod_{v\in S_{f}}(1-Nv^{-r}\chi^{-1}(v)).\mathrm{Fitt}_{O_{p}}(e_{\chi}(H^{0}(F, \mathbb{Q}_{p}/\mathbb{Z}_{p}(1-r))\otimes O_{p}))\mathrm{Fitt}^{-1}_{O_{p}}(e_{\chi}(K_{-2r}(O_{F, S})\otimes O_{p}))L^{'}(r,\chi^{-1})
\end{align*}
}
\end{itemize}
where $O_{F, S}$ refers to the ring of $S$-units of $F$, and $\rho_{F}^{r}$ denotes the extension of scalars of the Beilinson regulator over $O_p$. We also identified $e_{\chi}O_{p}[G]=e_{\chi}O_{p}$ with $O_{p}$.
\end{theorem}
\begin{proof}
Let $p$ be a rational prime and $S$ a finite set of places of $k$ containing the infinite places, the $p$-places and the places which ramify in $F/k$. We denote by $S_f$ the subset of finite places inside $S$.\\
Recall that $O$ denotes the ring of integers of $E$, and $O_{p}=O\otimes \mathbb{Z}_{p}$.\\
We also fix $T_p =O_{p}[G](r)$.\\
If we denote by $\vartheta^{r}_{F, S}$ the compositum map
$$\vartheta^{r}_{F, S}:= \mathcal{r}^{-1}\circ \vartheta^{r}_{F, S_{p}}$$
where $\mathcal{r} : \mathrm{Det}_{E_{p}[G]}R\Gamma_{c}(O_{k, S}, M_{p})\widetilde{\rightarrow}\mathrm{Det}_{E_{p}[G]}R\Gamma_{c}(O_{k, S_{p}}, M_{p})$ is the isomorphism obtained from the distinguished triangle (\ref{6}) by mapping $\otimes_{v\in S_{f}\backslash S_{p}}\mathrm{Det}_{E_{p}[G]}R\Gamma_{f}(k_{v}, M_p)$ to $E_{p}[G]$ via 
$\otimes_{v\in S_{f}\backslash S_{p}}\mathrm{Id}_{V_{v}, triv}$ (as explained in the proof of proposition \ref{proposition4.3}),
then the statement of the ETNC for the enlarged set of primes $S$ reads
$$O_{p}[G]\vartheta^{r}_{F, S} \circ (\vartheta^{r}_{F, \infty})^{-1} (L^{*}(E(r)_{F},0)^{-1})=\mathrm{Det}_{O_{p}[G]}R\Gamma_{c}(O_{k,S}, T_p)$$

For any subset of finite places $S^{'}$ of $S$
we write $$\varepsilon_{S^{'}}(r):=\prod_{v\in S^{'}}(1-Nv^{-r}\mathrm{Frob}_{v}e_{I_{v}})$$
where $\mathrm{Frob}_{v}$ is any representative in $G$ of the Frobenius map  associated to the finite place $v$, and $e_{I_{v}}:= \frac{\Sigma_{\delta \in I_v}\delta}{\# I_v}$ is the idempotent of $\mathbb{Q}[G]$ which corresponds to the inertia subgroup $I_v$.\\
In the proof of Proposition \ref{proposition4.3} we mentioned the equality (\cite{burnsFlach} , (11), (12))
$$\tilde{\vartheta}^{r}_{F, S_{p}}=\varepsilon_{S_p}(r)\vartheta^{r}_{F, S_{p}})$$
Again, by the same reasoning, we also have
$$\mathcal{r}=\varepsilon_{S_{f}\backslash S_{p}}(r).\tilde{\mathcal{r}}$$
Then 
\small{
\begin{align*}
&(\vartheta^{r}_{F, S})^{-1}\mathrm{Det}_{O_{p}[G]}R\Gamma_{c}(O_{k,S},T_{p})=((\vartheta^{r}_{F, S_p})^{-1}\circ \mathcal{r})(\mathrm{Det}_{O_{p}[G]}R\Gamma_{c}(O_{k,S},T_{p}))\\
=&(\varepsilon_{S_{p}}(r).(
\tilde{\vartheta}^{r}_{F, S_{p}})^{-1}\circ \varepsilon_{S_{f}\backslash S_{p}}(r).\tilde{\mathcal{r}})\mathrm{Det}_{O_{p}[G]}R\Gamma_{c}(O_{k,S},T_{p}))\\
=&\varepsilon_{S_{f}}(r)((\tilde{\vartheta}^{r}_{F, S_{p}})^{-1}\circ \tilde{\mathcal{r}})\mathrm{Det}_{O_{p}[G]}R\Gamma_{c}(O_{k,S},T_{p}))\\
=&\varepsilon_{S_f}(r)(\mathcal{H} \circ \jmath \circ \mho)(\mathrm{Det}_{O_{p}[G]}(R\Gamma(O_{k,S}, T_{p}^{\vee}(1))^{\vee}[-3])\otimes \mathrm{Det}^{-1}_{O_{p}[G]}(\prod_{\sigma:k\rightarrow \mathbb{C}}T_{p})^{+}))\\
=&\varepsilon_{S_f}(r)(\mathrm{Det}(ch^{\vee})\otimes \mathrm{Det}^{-1}(ch)\otimes  \mathrm{Det}^{-1}(cyc^{\vee})\otimes\mathrm{Det}(cyc)\otimes\mathrm{Det}^{-1}(\alpha))^{-1}(\mathrm{Det}_{O_{p}[G]}(R\Gamma(O_{k,S}, T_{p}^{\vee}(1))^{\vee}[-3])\\
&\otimes \mathrm{Det}^{-1}_{O_{p}[G]}(\prod_{\sigma:k\rightarrow \mathbb{C}}T_{p})^{+}))
\end{align*}
}
\normalsize{
In the last equality, the maps }$\mathrm{Det}(ch^{\vee})$, $\mathrm{Det}^{-1}(ch)$, $\mathrm{Det}^{-1}(cyc^{\vee})$ and $\mathrm{Det}(cyc)$  are the induced maps which  apply to $\mathrm{Det}_{O_{p}[G]}(R\Gamma(O_{k,S}, T_{p}^{\vee}(1))^{\vee}[-3])$ by composition with $\jmath \circ \mho$.\\


Note that, since we suppose $k$ to be imaginary quadratic, we have $(\prod_{\sigma:k\rightarrow \mathbb{C}}T_{p})^{+}= T_p =O_{p}[G](r)$.

Thus
\small{
\begin{align*}
&(\vartheta^{r}_{F, S})^{-1}\mathrm{Det}_{O_{p}[G]}R\Gamma_{c}(O_{k,S},T_{p})=\\
&=\varepsilon_{S_f}(r)(\mathrm{Det}(ch^{\vee})\otimes \mathrm{Det}^{-1}(ch)\otimes  \mathrm{Det}^{-1}(cyc^{\vee})\otimes\mathrm{Det}(cyc)\otimes\mathrm{Det}^{-1}(\alpha))^{-1}(\mathrm{Det}_{O_{p}[G]}(R\Gamma(O_{k,S}, T_{p}^{\vee}(1))^{\#}\otimes \mathrm{Det}^{-1}_{O_{p}[G]}T_p)\\
&=\varepsilon_{S_f}(r)\bigg(\Big((\mathrm{Det}(ch^{\vee})\otimes \mathrm{Det}^{-1}(ch)\otimes  \mathrm{Det}^{-1}(cyc^{\vee})\otimes\mathrm{Det}(cyc))^{-1}\mathrm{Det}_{O_{p}[G]}(R\Gamma(O_{k,S}, T_{p}^{\vee}(1))^{\#}\Big)\otimes \mathrm{Det}^{-1}_{O_{p}[G]}(\alpha^{-1}(T_{p}))\bigg)
\end{align*}
}
\normalsize{which is equivalent to}
\small{
\begin{align*}
&(\vartheta^{r}_{F, S})^{-1, \#}\mathrm{Det}_{O_{p}[G]}R\Gamma_{c}(O_{k,S},T_{p})^{\#}=\\
&=\varepsilon_{S_f}(r)^{\#}\bigg(\Big((\mathrm{Det}(ch^{\vee})\otimes \mathrm{Det}^{-1}(ch)\otimes  \mathrm{Det}^{-1}(cyc^{\vee})\otimes\mathrm{Det}(cyc))^{-1, \#}\mathrm{Det}_{O_{p}[G]}(R\Gamma(O_{k,S}, T_{p}^{\vee}(1))\Big)\otimes \mathrm{Det}_{O_{p}[G]}(\alpha^{-1}(T_{p}))^{\vee}\bigg)
\end{align*}
}
\normalsize{yet we have }
\begin{itemize}
\item $H^{0}(O_{k, S},T_{p}^{\vee}(1))=0$ since $r<0$ and $H^{3}(O_{k, S},T_{p}^{\vee}(1))=0$ .
\item Assuming that the Chern maps are isomorphisms on the integral level \cite{voevodsky}, we write 
$$ch^{O_p}_{r,k} : K_{k-2r}(O_{F,S})\otimes O_{p}\xrightarrow{\;\sim\;} H^{2-k}(O_{k,S}, O_{p}(1-r))$$ for the isomorphisms induced by tensoring with $\mathrm{Id}_{O_p}$ for $k=0,1$.

Then
\begin{align*}
&(\mathrm{Det}(ch^{\vee})\otimes \mathrm{Det}^{-1}(ch))^{-1, \#}(\mathrm{Det}^{-1}_{O_{p}[G]}H^{1}(O_{k,S}, T_{p}^{\vee}(1))\otimes \mathrm{Det}_{O_{p}[G]}H^{2}(O_{F,S}, T_{p}^{\vee}(1))\\
=&(\mathrm{Det}^{-1}((ch^{O_p}_{r,1})^{-1})\otimes \mathrm{Det}((ch^{O_p}_{r,0})^{-1}))(\mathrm{Det}^{-1}_{O_{p}[G]}H^{1}(O_{F,S}, O_{p}(1-r))\otimes \mathrm{Det}_{O_{p}[G]}H^{2}(O_{F,S}, O_{p}(1-r)))\\
=&\mathrm{Det}^{-1}_{O_{p}[G]}((ch^{O_p}_{r,1})^{-1}H^{1}(O_{F,S}, O_{p}(1-r)))\otimes \mathrm{Det}_{O_{p}[G]}((ch^{O_p}_{r,0})^{-1}H^{2}(O_{F,S}, O_{p}(1-r)))\\
=&\mathrm{Det}^{-1}_{O_{p}[G]}(K_{1-2r}(O_{F,S})\otimes O_{p})\otimes \mathrm{Det}_{O_{p}[G]}(K_{-2r}(O_{F,S})\otimes O_{p})
\end{align*}
\item One also has
\begin{align*}
\mathrm{Det}_{O_{p}[G]}(\alpha^{-1}(T_{p}))^{\vee}&\cong \mathrm{Det}_{O_{p}[G]}(O_{p}[G])\\
&=O_{p}[G].
\end{align*}
\end{itemize}

Conjecture \ref{conjecture4.3.4}.4 is equivalent to
\small{
\begin{align*}
(L^{*}(E(r)_{F},0)^{-1})^{\#}&=\varepsilon_{S_f}(r)^{\#}.(\vartheta^{r}_{F, \infty})^{\#}\bigg(\Big(\mathrm{Det}^{-1}_{O_{p}[G]}(K_{1-2r}(O_{F,S})\otimes O_{p})\otimes \mathrm{Det}_{O_{p}[G]}(K_{-2r}(O_{F,S})\otimes O_{p})\Big)\bigg)\\
&=\prod_{v\in S_{f}}(1-Nv^{-r}\mathrm{Frob}_{v}^{-1}e_{I_{v}}).(\vartheta^{r}_{F, \infty})^{\#}\bigg(\Big(\mathrm{Det}^{-1}_{O_{p}[G]}(K_{1-2r}(O_{F,S})
\otimes O_{p})\\ &\otimes \mathrm{Det}_{O_{p}[G]}(K_{-2r}(O_{F,S})\otimes O_{p})\Big)\bigg)
\end{align*}
}

\normalsize{and since} $p$ does not divide $\mid{G}\mid$ and $O_{p}$ contains all values of characters of $G$, the latter is equivalent to the following being true for all characters $\chi$ of $G$
\begin{align*}
L^{'}(r,\chi^{-1})^{-1}&=\prod_{v\in S_{f}}(1-Nv^{-r}\chi^{-1}(v)).(\vartheta^{r}_{F, \infty})^{\#}\bigg(\Big(\mathrm{Det}^{-1}_{O_{p}}(e_{\chi}(K_{1-2r}(O_{F,S})
\otimes O_{p}))\\ &\otimes \mathrm{Det}_{O_{p}}(e_{\chi}(K_{-2r}(O_{F,S})\otimes O_{p}))\Big)\bigg)
\end{align*}
The ring $O_p$ is a product of discrete valuation rings, which means that every finite module is of projective dimension less or equal to $1$ over $O_p$ , hence
\begin{align*}
\mathrm{Det}_{O_{p}}(e_{\chi}(K_{1-2r}(F)\otimes O_{p}))&=\mathrm{Det}_{O_{p}}(e_{\chi}(K_{1-2r}(F)\otimes O_{p})_{/tors}))\otimes \mathrm{Det}_{O_{p}}(e_{\chi}(tors_{O_{p}}(K_{1-2r}(F)\otimes O_{p})))\\
&=\mathrm{Det}_{O_{p}}(e_{\chi}(K_{1-2r}(F)\otimes O_{p})_{/tors}))\otimes \mathrm{Fitt}^{-1}_{O_{p}}(e_{\chi}(H^{0}(F, \mathbb{Q}_{p}/\mathbb{Z}_{p}(1-r))\otimes O_{p})).
\end{align*}
The map $(\vartheta^{r}_{F, \infty})^{\#}$ applies to  $\mathrm{Det}_{O_{p}}(e_{\chi}(K_{-2r}(O_{F,S})\otimes O_p))$ and $\mathrm{Det}_{O_{p}}(e_{\chi}(tors_{O_{p}}(K_{1-2r}(F)\otimes O_{p})))$ via identity since $K_{-2r}(O_{F,S})\otimes E$ and $tors_{O}(K_{1-2r}(F)\otimes O)\otimes E$ are trivial.\\
 It also
maps $(\mathrm{Det}^{-1}_{O_{p}}(e_{\chi}(K_{1-2r}(O_{F,S})
\otimes O_{p})_{/tors})$ by definition to $(\mathrm{Det}^{-1}_{O_{p}}\rho_{F}^{r}(e_{\chi}(K_{1-2r}(O_{F,S})
\otimes O_{p})_{/tors})$ (as explained in the proof of Theorem \ref{theorem4.4}).

We then get
\small{
\begin{align*}
&\rho_{F}^{r}(e_{\chi}(K_{1-2r}(F)_{/tors}\otimes O_{p}))=\\ &\prod_{v\in S_{f}}(1-Nv^{-r}\chi^{-1}(v)).\mathrm{Fitt}_{O_{p}}(e_{\chi}(H^{0}(F, \mathbb{Q}_{p}/\mathbb{Z}_{p}(1-r))\otimes O_{p}))\mathrm{Fitt}^{-1}_{O_{p}}(e_{\chi}(K_{-2r}(O_{F, S})\otimes O_{p}))L^{'}(r,\chi^{-1})
\end{align*}
$\Box$
}
\end{proof}
In \cite{Leung}, the author suggests a proof of the ETNC for abelian extensions over an imaginary quadratic field for odd primes whenever a certain condition is fulfilled (\cite{Leung}, Theorem 1.1).\\
This motivates the following corollary
\begin{corollary}\label{corollary4.9}
Let $F/k$ be a finite abelian extension of number fields with $k$ imaginary quadratic. Suppose that $p$ is a rational prime which is split and does not divide $6.\# G$ and that $r<0$. 
Then 
\small{
\begin{align*}
&\rho_{F}^{r}(e_{\chi}(K_{1-2r}(F)_{/tors}\otimes O_{p}))=\\ &\prod_{v\in S_{f}}(1-Nv^{-r}\chi^{-1}(v)).\mathrm{Fitt}_{O_{p}}(e_{\chi}(H^{0}(F, \mathbb{Q}_{p}/\mathbb{Z}_{p}(1-r))\otimes O_{p}))\mathrm{Fitt}^{-1}_{O_{p}}(e_{\chi}(K_{-2r}(O_{F, S})\otimes O_{p}))L^{'}(r,\chi^{-1})
\end{align*}
}
where $O_{F, S}$ refers to the ring of $S$-units of $F$, and $\rho_{F}^{r}$ denotes the extension of scalars of the Beilinson regulator over $O_p$.
\end{corollary}
\begin{proof}
This follows from Theorem \ref{theorem4.8} and (\cite{Leung}, Corollary 1.2).
$\Box$
\end{proof}
This implies the following
\begin{corollary}\label{corollary4.10}
Let $F/k$ be a finite abelian extension of number fields with $k$ imaginary quadratic, and $r<0$. Suppose that $O$ is the ring of integers of the extension generated over $\mathbb{Q}$ by all values of characters of $G:=\mathrm{Gal}(F/k)$, and let $S$ be any finite set of places containing the infinite places and the places which ramify in $F/k$.\\
Then Conjecture \ref{conjecture3.1} holds for the set of data $(F/k, r, S, O_p)$, for all rational primes $p$ which are split and such that $p\nmid 6\#G$.
\end{corollary}
\begin{proof}
Let $p$ be a rational prime which is split and doesn't divide $6.\#G$. Corollary \ref{corollary4.5} ensures that the first statement of Conjecture \ref{conjecture3.1} is fulfilled for $(F/k, r, S, O_p)$.\\
By Corollary \ref{corollary4.9} there exists at least one element $\epsilon\in K_{1-2r}\otimes O_p$ such that
\begin{align*}
\rho^{r}_{F,p}(\epsilon)&\in \prod_{v\in S_{f}\cup S_p}(1-Nv^{-r}\chi^{-1}(v)).\mathrm{Fitt}_{O_{p}}(e_{\chi}(H^{0}(F, \mathbb{Q}_{p}/\mathbb{Z}_{p}(1-r))\otimes O_{p}))L^{'}(r,\chi^{-1})
\end{align*}
where $S_f$ is the subset of finite places in $S$, and $S_p$ is the set of $p$ places of $k$. 
Since $e_{\chi}(H^{0}(F, \mathbb{Q}_{p}/\mathbb{Z}_{p}(1-r))\otimes O_{p})$ is invariant under the action of $\mathrm{Gal}(F/F^{\mathrm{ker}(\chi)})$, we get
\begin{align*}
e_{\chi}(H^{0}(F, \mathbb{Q}_{p}/\mathbb{Z}_{p}(1-r))\otimes O_{p})&\subset (H^{0}(F, \mathbb{Q}_{p}/\mathbb{Z}_{p}(1-r))\otimes O_{p})^{\mathrm{Gal}(F/F^{\mathrm{ker}(\chi)})}\\
&=H^{0}(F^{\mathrm{ker}(\chi)}, \mathbb{Q}_{p}/\mathbb{Z}_{p}(1-r))\otimes O_{p}
\end{align*}
Since the two modules $e_{\chi}(H^{0}(F, \mathbb{Q}_{p}/\mathbb{Z}_{p}(1-r))\otimes O_{p})$ and $H^{0}(F^{\mathrm{ker}(\chi)}, \mathbb{Q}_{p}/\mathbb{Z}_{p}(1-r))\otimes O_{p}$ are finite and $O_p$ is the direct sum of principal rings we get
$$\mathrm{Fitt}_{O_p}(e_{\chi}(H^{0}(F, \mathbb{Q}_{p}/\mathbb{Z}_{p}(1-r))\otimes O_{p}))\mid \mathrm{Fitt}_{O_p}(H^{0}(F^{\mathrm{ker}(\chi)}, \mathbb{Q}_{p}/\mathbb{Z}_{p}(1-r))\otimes O_{p})=w_{1-r}(F^{\mathrm{ker}(\chi)})O_p$$
This means that there exists elements $a$ and $b$ in $O_p$, such that $\mathrm{Fitt}_{O_p}(e_{\chi}(H^{0}(F, \mathbb{Q}_{p}/\mathbb{Z}_{p}(1-r))\otimes O_{p}))=aO_p$, and $w_{1-r}(F^{\mathrm{ker}(\chi)})=ab$.\\
We have then
$$\rho^{r}_{F,p}(b\epsilon)\in \prod_{v\in S_{f}\cup S_p}(1-Nv^{-r}\chi^{-1}(v)).w_{1-r}(F^{\mathrm{ker}(\chi)})L^{'}(r,\chi^{-1})O_p$$
Further, since
$$\prod_{v\in S_p}(1-Nv^{-r}\chi^{-1}(v))\equiv 1\;\mathrm{mod}\;p,$$
the element $\prod_{v\in S_p}(1-Nv^{-r}\chi^{-1}(v))$ is a unit in $O_p$, and if we write $\epsilon^{'}=b\epsilon$, we also have
$$\rho^{r}_{F,p}(\epsilon^{'})\in \prod_{v\in S_{f}}(1-Nv^{-r}\chi^{-1}(v)).w_{1-r}(F^{\mathrm{ker}(\chi)})L^{'}(r,\chi^{-1})O_p.$$
Yet,
$$\prod_{v\in S_{f}}(1-Nv^{-r}\chi^{-1}(v))L^{'}(r,\chi^{-1})=L^{'}_{S}(r, \chi^{-1})$$
Hence
$$\rho^{r}_{F,p}(\epsilon^{'})\in w_{1-r}(F^{\mathrm{ker}(\chi)})L^{'}_{S}(r,\chi^{-1})O_p$$
By Theorem \ref{theorem4.8}, we can choose the element $\epsilon^{'}$, so as to exactly have the following
$$\rho^{r}_{F,p}(\epsilon^{'})= w_{1-r}(F^{\mathrm{ker}(\chi)})L^{'}_{S}(r,\chi^{-1})\mid G \mid$$
Note that, since, in Theorem \ref{theorem4.8}, we made the identification $e_{\chi}O_p=O_p$, the exact statement is
$$\rho^{r}_{F,p}(\epsilon^{'})= w_{1-r}(F^{\mathrm{ker}(\chi)})L^{'}_{S}(r,\chi^{-1})\mid G \mid e_\chi$$
The result ensues.\hspace*{\fill}$\Box$
\end{proof}

\noindent\textbf{Acknowledgement}

\noindent In ($\cite{SaadElBoukhari}$, Theorem 6.3) the second author introduced an element $$Q(0):=\prod_{\ell\mid N, \ell\not = p}(1-\mathrm{Fr}_{\ell}^{-1}\ell^{m-1})$$
where, for a finite abelian extension $F/\mathbb{Q}$ , we denote by $N$ the conductor of $F$ and by $m$ an odd integer such that $m\geq 3$. We write $\mathrm{Fr}_{\ell}$ for the Frobenius automorphism at the prime $\ell$.\\
However, in the latter expression, the eulerian factors $1-\mathrm{Fr}_{\ell}^{-1}\ell^{m-1}$  should be replaced by $(1-\mathrm{Fr}_{\ell}^{-1}\ell^{m-1})e_{I_\ell}\in \mathbb{Z}[G]$. Here $I_\ell$ is the inertia subgroup of $\mathrm{Gal}(F/\mathbb{Q})$ corresponding to the prime $\ell$, $e_{I_\ell}:=(\#I_{\ell})^{-1}\Sigma_{g\in I_{\ell}}g$ and $\mathrm{Fr}_{\ell}$ denotes a representative of the Frobenius map in $G:=\mathrm{Gal}(F/\mathbb{Q})$.\\ This corrects the definition given in \cite{SaadElBoukhari} and has been suggested by the authors in \cite{BurnsDavidSano} who we particularily thank for their observation.

1- J. Assim, Moulay Ismail University of Meknès, Department of Math., B.P. 11201 Zitoune, Meknès, Morocco.
\textit{E-mail:} j.assim@fs.umi.ac.ma\\

2- S. El Boukhari, Moulay Ismail University of Meknès, Department of Math., B.P. 11201 Zitoune, Meknès, Morocco.
\textit{E-mail:} saadelboukhari1234@gmail.com

\end{document}